\documentclass[11pt]{article}
\usepackage{amsfonts}
\usepackage{amssymb,amsmath,amsthm, enumerate}
\usepackage{latexsym}
\usepackage{fullpage}
\usepackage{hyperref}
\usepackage{graphicx}
\newtheorem{theorem}{Theorem}[section]

\newtheorem{prop}[theorem]{Proposition}

\newtheorem{lemma}[theorem]{Lemma}
\newtheorem{cor}[theorem]{Corollary}

\theoremstyle{definition}
\newtheorem*{remark}{Remarks}

\newcounter{tenumerate}

\def\P{\mathbb{P}}

\newcommand{\one}{\1}

\renewcommand{\epsilon}{\varepsilon}

\newcommand{\1}{\mathbf{1}}

\newcommand{\Aidekon}{{A\"{i}dekon }}

\newcommand{\E}{{\mathbb E}}
\newcommand{\remove}[1]{}
\renewcommand{\le}{\leqslant}
\renewcommand{\ge}{\geqslant}
\renewcommand{\leq}{\leqslant}
\renewcommand{\geq}{\geqslant}

\def\XXint#1#2#3{{\setbox0=\hbox{$#1{#2#3}{\int}$}
\vcenter{\hbox{$#2#3$}}\kern-.5\wd0}}

\begin{document}

\title{Convergence in law of the maximum of 
nonlattice branching random walk}
\author{Maury Bramson\\
University of Minnesota
\thanks{Partially supported by NSF grants DMS-1105668 and DMS-1203201.} \and
Jian Ding \\
University of Chicago\thanks{Partially supported by NSF grant DMS-1313596.} \and Ofer Zeitouni\thanks{Partially supported by
NSF grant
DMS-1106627, a grant from the Israel Science
Foundation, and the Herman P. Taubman chair of Mathematics at the
Weizmann institute.}
\\  Weizmann institute}

\date{April 1, 2015}
\maketitle

\begin{abstract}
Let $\eta^*_n$ denote the maximum, at time $n$, of a nonlattice
one-dimensional 
branching 
 random walk $\eta_n$ possessing (enough)
exponential moments. In a seminal paper,
\Aidekon \cite{Aidekon} demonstrated 
convergence of $\eta^*_n$ in law, after recentering, and gave a representation
of the limit.  
We give here a shorter proof of this convergence
by employing reasoning motivated by Bramson, Ding and Zeitouni \cite{BDZ}.  
Instead of spine methods and a careful analysis of the renewal measure
for killed
random walks, our approach employs a modified version of the second moment
method that may be of independent interest. We indicate the modifications needed in order to handle lattice
random walks.

Soit $\eta_n^*$ le maximum, au moment $n$,
d'une marche al\'{e}atoire branchante 
unidimensionelle
qui n'est pas support\'{e} sur un r\'{e}seau et qui poss\`{e}de suffisament
de moments exponentiels.
Dans un article fondateur, \Aidekon \cite{Aidekon} a 
demontr\'{e} la convergence de $\eta_n^*$, apr\`{e}s recentrage, en distribution, et a donn\'{e}
une repr\'{e}sentation de la limite. Nous donnons ici une preuve plus courte de cette convergence en employant un raisonement motiv\'{e}
par Bramson, Ding et Zeitouni \cite{BDZ}.
Au lieu des methodes spinales et d'une analyse de la mesure de renouvellement pour la marche al\'{e}atoire tu\'{e},  notre m\'{e}thode
utilise une version modifi\'{e}e de la m\'{e}thode du deuxi\`{e}me moment, qui peut \^{e}tre d'int\'{e}r\^{e}t
ind\'{e}pendant. Nous indiquons les modifications n\'{e}cessaire pour traiter les marches al\'{e}atoires sur un r\'{e}seau.
\end{abstract}

\section{Introduction}
We consider nonlattice
one-dimensional 
branching random walk (BRW), $\{\eta_n\}_{n=0,1,2,\ldots}$, 
with offspring distribution
$\{p_i\}_{i=1,2,\ldots}$ and 
random walk increments $\{w(dy)\}_{y\in \mathbb{R}}$.  
The BRW is constructed in the usual inductive manner using 
$p_{\cdot}$ and $w(\cdot)$, with individuals of the $n$th generation moving
independently of each other according to $w(\cdot)$, 
from the site of their parent in the $(n-1)$st generation.
We denote by $\rho$ the mean of $p_{\cdot}$, by 
$\gamma_0$ the mean of
$w(\cdot)$, assume that $p_{\cdot}$ has finite second
moment, and that $w(\cdot)$ is nonlattice and
has exponential moments in an appropriate interval  
(which will be specified shortly), using the notation
\begin{equation}
\label{eq1.1}
K =\sum_i i^2p_i , \quad
\phi (\theta)=\int{e^{\theta y}w(dy)} \,. 
\end{equation}
(Here, nonlattice means that the support of $w(\cdot) + y$ 
is not contained in 
any discrete subgroup of $\mathbb{R}$ for any $y$.)
We denote by $V_n$ the set of $n$th generational offspring, with $\eta_{v,n}, v\in V_n$, being the positions of
these offspring, and set $\eta^*_n =\max_{v\in V_n}\eta_{v,n}$.

The limiting behavior of $\eta^*_n$, as $n\rightarrow\infty$, has been studied
since the early 1970s.  A strong law of large numbers for $\eta^*_n/n$ was 
first given in Kingman \cite{King}; see \Aidekon \cite{Aidekon} for 
general literature on the subject of branching random walk.   
In his  recent seminal paper,
\Aidekon \cite{Aidekon} has shown the sharp result that 
$\eta^*_n - (c_1 n - c_2 \log n)$ converges in distribution 
for appropriate $c_1$, $c_2$, which
depend on $p_{\cdot}$ and $w(\cdot)$; he also identified the limit as a 
Gumbel distribution shifted by a particular random variable,
the limit of the derivative martingale of the branching random walk.

The behavior of $\eta^*_n$ is related to the limiting 
behavior of the maximum of 
branching Brownian motion.  The latter problem traces its roots back to 
Kolmogorov, Petrovsky, and Piscounov \cite{KPP}  
and Fisher \cite{Fish}; sharp
results were obtained in Bramson \cite{BrMem}, and an identification
of the limit as a Gumbel distribution shifted by the derivative martingale
was obtained by
Lalley and Sellke \cite{LS}.  Results comparable to those in \cite{BrMem}
were obtained in the context of the two-dimensional discrete 
Gaussian free field in Bramson, Ding, and Zeitouni \cite{BDZ}.   
Here, we employ reasoning related to that in 
the last paper
to show convergence in distribution of 
$\eta^*_n$ after recentering, and to identify the limit.  

To state our main result, Theorem \ref{thm1.2}, we first introduce the
following terminology.  Let $I(\cdot)$ denote the rate function for
$w(\cdot)$, that is, for $\lambda > \gamma_0$,
\begin{equation}
\label{eq1.1newer}
 I(\lambda) = \sup_{\theta > 0}[\theta\lambda - \log \phi (\theta)]\,.
\end{equation}
Assume that
\begin{equation}
\label{eq1.2}
\log \rho \in \, \{I(\cdot)\}^\circ \,,\quad
c_1\in \, {\left\{(\log \phi)'(\cdot)\right\}}^\circ,
\end{equation}
where $c_1$ satisfies $I(c_1) = \log \rho$ 
(and $G^\circ$ 
denotes the interior of  $G$).
Then, 
$I(\cdot)$ is convex and differentiable in a neighborhood of $c_1$. 
Denote by $\bar{\theta}$ the value of $\theta$ 
at which the supremum in (\ref{eq1.1newer}) is taken for
$\lambda = c_1$, and set
$c_2 = 3/2\bar{\theta}$.  We then set $m_n = c_1 n - c_2 \log n$.
Also, set
\begin{equation*}
Z_k = \sum_{v\in V_k}(c_1k - \eta_{v,k})
\mathrm{e}^{-\bar{\theta}(c_1k - \eta_{v,k})}\,, 
\end{equation*}
and denote by $\mathcal{F}_k$ the $\sigma$-algebra generated by
the BRW up through time $k$.  

Our main result is the following theorem.
\begin{theorem}
\label{thm1.2}
Assume that $\eta_n$ is a nonlattice branching random walk
satisfying (\ref{eq1.2}), with $K < \infty$.  Then, 
$\eta^*_n - m_n$ converges in law 
as $n\rightarrow\infty$.  Moreover,
$Z = \lim_{k\rightarrow\infty}Z_k$
exists and is finite and positive with probability $1$, and there exists a
constant $\alpha ^*>0$ so that, for each $z\in \mathbb{R}$,
\begin{equation}
\label{eq1.4}
\lim_{k\rightarrow\infty}\lim_{n\rightarrow\infty}
\P(\eta^*_n \le m_n +z |\,\mathcal{F}_k)
= \exp\{-\alpha ^*Z \mathrm{e}^{-\bar{\theta}z}\} \quad \text{a.s.}
\end{equation}
\end{theorem}
\begin{remark}
  \begin{enumerate}
    \item Theorem \ref{thm1.2} is the analog of Theorem 1.1 
of \Aidekon \cite{Aidekon}. 
The latter paper has nearly optimal conditions on the branching and random
walk distributions, which we have not tried to duplicate here. 
\item
Our proof of Theorem \ref{thm1.2} 
is, we believe,
shorter and more elementary than that in \cite{Aidekon}, 
employing techniques developed in Ding and Zeitouni \cite{DiZ}
and Bramson, Ding, and Zeitouni \cite{BDZ}.
In particular, we do not use the 
convergence of the derivative martingale in the convergence in law proof, 
we do not 
use renewal theory (except to the extent that certain estimates from
random walk, developed in \cite{Car}, are used), and we 
do not work with the spine representation. Instead, we employ a variant
of the second moment method that is tailored toward deriving tail estimates
and involves a truncation that keeps only the leading particle
in each subtree of depth $k$ rooted at a vertex in $V_{n-k}$.
\item The result (\ref{eq1.4}) for branching 
Brownian motion dates back to Lalley and Sellke \cite{LS} 
and states that the limit can be
written as a random shift (by the limit of the so called
derivative martingale) of the 
Gumbel  distribution.  
\item When the first part of  \eqref{eq1.2} does not
hold (which is only possible if the support of
$w(\cdot)$ has a finite upper bound), 
non-standard centering and limit behavior is possible for 
$\eta_n^*$ (see, for example,
Bramson \cite{BM78}).
The second part of \eqref{eq1.2}  
ensures that, after an exponential change of measure that recenters
the measure at $c_1$, the resulting measure still possesses
exponential moments. 
\item We believe that the approach discussed in this paper allows one to 
also handle the case of lattice BRWs. We discuss this extension
in Section \ref{sec-lattice}.
\end{enumerate}
\end{remark}
An important part of the demonstration of Theorem \ref{thm1.2} involves 
showing that 
$\P(\eta_n^* - m_n > z) \sim \alpha^*ze^{-\bar{\theta} z}$ 
for large $z$, which is done in 
Proposition \ref{prop-limiting-tail-gff}.
(Here and later, we write $a_n(z)\sim b_n(z)$ if
$\lim_{z\to\infty}\limsup_{n\to\infty} a_n(z)/b_n(z)=
\lim_{z\to\infty}\liminf_{n\to\infty} a_n(z)/b_n(z)=1$.)
The long Section \ref{sec-limittail} is
devoted to showing this proposition, with the two main steps being Propositions
\ref{prop-gff-first-moment-dictates} and \ref{prop-asymptotic-first-moment}.  Proposition \ref{prop-gff-first-moment-dictates} compares
$\P(\eta_n^*- m_n > z)$ with an appropriate expectation corresponding to the number of
particles present at a time $n-\ell$, 
$\ell \ll n$,
that lie below a given boundary until then and that have at least one offspring above
$m_n +z$ at time $n$; 
related estimates are also present in \cite{Aidekon}.
The second moment estimates used here (in Proposition \ref{lem-second-moment}) 
are a more refined version of 
those used elsewhere in the branching literature. 
Proposition \ref{prop-asymptotic-first-moment}
then shows that this expectation is approximated by 
$\alpha^*ze^{-\bar{\theta} z}$.  

In the proof of Theorem \ref{thm1.2}, we divide the
evolution of $\{\eta_j\}_{0\le j\le n}$ into two time intervals, $[0,k]$ and $[k,n]$, first
letting $n\rightarrow\infty$ and then $k\rightarrow\infty$.  
At time $k$, we decompose the process 
$\{\eta_j\}_{j=0,1,2,\ldots}$ into $|V_{k}|$ processes, 
each given by a BRW 
$\{\eta^{v'}_{v,j}\}_{j=0,1,2,\ldots}$
 descending from $v' \in V_{k}$ and restarted at 
position $0$, i.e., 
\begin{equation}
\label{eq2.1}
\eta^{v'}_{v,j} = \eta_{v,j+k} - \eta_{v',k} \quad \text{for } 
v\in V^{v'}_j,
\end{equation}
where $V^{v'}_j = V^{v',k}_j $ denotes the set of $j$th generation descendents of $v^{\prime}$ 
in the $(j+k)$th generation of the BRW;
the processes $\{\eta^{v^{\prime}}_{j}\}_{j=0,\ldots,n'}$ will be
independent copies of $\{\eta_{j}\}_{j=0,\ldots,n'}$.

The first part of Theorem \ref{thm1.2} follows quickly
from Proposition \ref{prop-limiting-tail-gff}
together with the decomposition in (\ref{eq2.1}).
The limit (\ref{eq1.4}) of Theorem \ref{thm1.2} employs reasoning similar to that for Theorem 1 
of Lalley and Sellke \cite{LS}.  
Both results are proved in Section \ref{sec-maintheorem}.

In Section \ref{sec-prelim}, various technical results are 
demonstrated that will be needed in Sections \ref{sec-limittail}
and \ref{sec-maintheorem}.  Basic tools
for these results are the crossing probabilities for random walks of certain 
curves, which are given in Lemmas \ref{lemmaABRold} -- \ref{prop-1DRW}, with the proof of the last two being 
deferred to the appendix.

\noindent{\bf Notation.}
For functions $F(\cdot)$ and $G(\cdot)$, we write $F \lesssim G$ or $F = O(G)$
if there exists an absolute constant $C>0$ such that 
$F \leq C G$ everywhere in the domain, and $F \asymp G$ if 
$F \lesssim G$ and $G\lesssim F$.  We sometimes abbreviate $F(x) = o_x(1)$  
if $F(x) \rightarrow_{x\rightarrow\infty} 0$.
For functions $F,G$ of a real or integer variable, we write
$F\sim G$ if $F/G$ converges to $1$ as the variable tends to
infinity. Finally,
for $x\in \mathbb{R}$, $\lfloor x \rfloor$ denotes the
largest integer not greater than $x$.

\section{Preliminaries}
\label{sec-prelim}

\subsection{Some random walk inequalities}
\label{subsec-longlemma}

In this subsection, we state two lemmas, Lemma \ref{lemmaABR}  and Lemma \ref{prop-1DRW}, that give
bounds on the probability of mean zero random walks 
not crossing specified curves.  These lemmas will be applied repeatedly
in this section and the next.
  
For both lemmas, we will employ a version of the ballot
theorem that is a slight modification of that given in Theorem 1 of 
Addario-Berry and Reed \cite{ABR}.   Here, $\{X_k\}_{k=1,2,\ldots}$ denote
independent copies of mean zero random variables $X$, and 
$S_n = \sum_{k=1}^n X_k$; $X$ will also be assumed to be nonlattice.
\begin{lemma}
\label{lemmaABRold}
In addition to the above assumptions, assume that $X$ has finite  variance. 
For all $n$ and all $a,y\ge 0$, $b>a$,  there is a $C>0$, depending only
on the law of $X$ and on $b-a$, such that
\begin{equation}
\label{equationABR1}
\P(S_n\in(a, b), S_k >-y \text{ for all } 0<k<n) \le \frac{C (y\vee 1)((y+a)\vee 1)}{n^{3/2}}
\end{equation}
and such that, for all $a$ with $0\le a\le \sqrt{n}$,
\begin{equation}
\label{equationABR2}
\P(S_n \in (a, b), S_k >0 \text{ for all } 0<k<n) \ge \frac{(a\vee 1)}{Cn^{3/2}}.
\end{equation}

\end{lemma}

Lemma \ref{lemmaABRold}  differs from Theorem 1 of 
Addario-Berry and Reed \cite{ABR} only in that (\ref{equationABR1})
is phrased here for general $y>0$, rather than just for $y=0$ as in the paper.  
The proof of (\ref{equationABR1})
remains essentially the same as in \cite{ABR}:
The time interval $[0,n]$ is divided into three parts, $[0,n/4]$,
$[n/4,3n/4]$, and $[3n/4,n]$.  
For both the first and third subintervals, Lemma 3 (iii) of \cite{ABR}, 
which gives an upper bound on the first time at which $S_k < -y'$, $y'\ge 0$, is applied in its
general form, rather than being restricted to $y'=0$ for the 
first subinterval as in the paper.  As in \cite{ABR},
for the middle term, one employs an upper bound on the density of 
$S_{3n/4} - S_{n/4}$.   The three upper bounds are 
then multiplied together to give
(\ref{equationABR1}).

In our applications, the above random walk $\{S_k\}_{k=0,\ldots,n}$ will correspond to the
random walk obtained by first subtracting $c_1 k$ from the random walk
underlying our BRW, and then tilting the corresponding measure 
so that the mean of the random walk associated with the tilting
is $0$.   Since Theorem \ref{thm1.2} instead requires the
nonlinear centering $m_n$ at time $n$,  
which differs from $c_1 n$ by $c_2 \log n$, 
we will in practice apply the following perturbation of
Lemma \ref{lemmaABRold}, which instead bounds the random walk
$\{S_k^{(n)}\}_{k=0,\ldots,n}$ defined below.
Note that, for $d^{(n)}=0$, $S_k^{(n)} \stackrel{d}{=} S_k$.

\begin{lemma}
\label{lemmaABR}
Let $X$ and $S_n$ be as above, and in addition assume that 
$\E (\mathrm{e}^{\theta X}) < \infty$, for 
$|\theta| \le \theta_0$ for some $\theta_0 >0$. 
Set $S_k^{(n)} = \sum_{i=1}^k X_i^{(n)}$, where 
$X_i^{(n)} = X_i + d^{(n)}$.  Assume that either $d^{(n)} >0$ for all $n$, that
$d^{(n)} <0$ for all $n$, or that $d^{(n)} \equiv 0$, with in each case $|d^{(n)}| \le c(\log n)/n$ for some
$c>0$.  Define the probability measure $\P^{(n)}$, on paths in $[0,n]$, by
\begin{equation}
\label{eq-change-of-measure-0}
\frac{d\mathbb{P}^{(n)}}{d\mathbb{P}} = 
\frac{\mathrm{e}^{-\theta^{(n)}S_n}}
{\mathbb{E}({\mathrm{e}^{-\theta^{(n)}S_n }})},
\end{equation}
with $\theta^{(n)}$ being chosen so that
$\mathbb{E}^{(n)}(X_1^{(n)}) = 0$.
Then $S_k^{(n)}$ satisfies the analogs of 
(\ref{equationABR1}) and (\ref{equationABR2}), 
with $\mathbb{P}^{(n)}$ replacing $\mathbb{P}$ and the
constants $C$ depending on $c$.
(We will refer to these inequalities as (\ref{equationABR1}) and (\ref{equationABR2})
as well.)
\end{lemma}

Lemma \ref{lemmaABR}, together with the following lemma, will be proved
in the appendix. 
Here, $h(\cdot)$ is a non-negative function such 
that $h(0)=0$, and $h(n) \leq C' \log (n+1)$ for a given constant 
$C'>0$ and all $n\in \mathbb{Z}_+$.  
\begin{lemma}\label{prop-1DRW}
Let $X$, $S_n$, $S_k^{(n)}$, $\P^{(n)}$, $d^{(n)}$, and $c$ be as in Lemma \ref{lemmaABR}.  
For any $y\geq 1$ and $-y+1\leq a <b<\infty$, 
there exists $\beta_{y, a, b} > 0$ such that 
\begin{equation}
\label{eq-S-y}
\lim_{n\rightarrow\infty} n^{3/2}\P^{(n)}(S_n^{(n)}\in (a, b), 
S_k^{(n)}\geq -y \mbox { for all } 0 < k < n) = 
\beta_{y, a, b} 
\end{equation}
with, for some $\beta^* > 0$, 
\begin{equation}
\label{eq-S-z}
\lim_{y,y+a \rightarrow\infty} \beta_{y, a, b}/(b-a)y(y+a) = \beta^*
\end{equation} 
if $b-a > 0$ is fixed as $y,a\rightarrow\infty$;
$\beta_{y,a,b}$ is continuous in $a$ and $b$ and right continuous in $y$.
Furthermore, there exists $\delta_{\bar{y}}$, 
with $\delta_{\bar{y}}\searrow 0 $ as 
$\bar{y}:= y  \wedge (y +a) \nearrow \infty$, such that, 
for $-y +1\le a<b< \infty$, 
\begin{equation}\label{eq-S-y-gamma}
\limsup_{n\to \infty} n^{3/2}\P^{(n)}(S_n^{(n)}\in (a, b), S_k^{(n)}\geq -y -  
y^{1/10}-h(k \wedge (n-k)) \mbox { for all } 0 < k < n) \leq \beta_{y, a, b} (1+ \delta_{\bar{y}})\,.
\end{equation}
If, in addition, $h(\cdot)$ is 
increasing
and concave, then, for fixed $\varepsilon= b-a > 0$,
 there exist $C>0$ and $n_{\varepsilon} \in \mathbb{Z}_+$ such that,
for $n\ge n_{\varepsilon}$, $ n/2 \le j \le n$, $y\ge 1$, 
and $-y - h(n-j) +1\le a < b< \infty$,
\begin{equation}
\label{eq-S-y-new}
n^{3/2}\P^{(n)}(S_j^{(n)}\in (a, b), S_k^{(n)}\geq -y - h(k \wedge (n-k)) \mbox { for all } 
0 < k < j) \leq Cy(y+a + h(n-j))\,.
\end{equation}
If $S_k^{(n)} \ge -y$ is replaced by the strict inequality $S_k^{(n)} > -y$ in (\ref{eq-S-y}), and the
analogous change is made in (\ref{eq-S-y-gamma}), then the analogs of (\ref{eq-S-y}) and 
(\ref{eq-S-y-gamma}) continue to hold for appropriate $\beta_{y,a,b}^o \,$, which is continuous in $a$ 
and $b$, and left continuous in $y$.  
None of the terms $\beta_{y,a,b}$, $\beta_{y,a,b}^o$, $\delta_{\bar{y}}$, 
$C$ and $n_{\varepsilon}$ depends on $d^{(n)}$, for fixed $c$.
\end{lemma}
It follows from (\ref{eq-S-y}) and (\ref{eq-S-y-gamma})  that
the ratio of the probabilities in these two displays lies within
$[1,1+2\delta_y]$ for given $y$, $a\ge 0$, and large enough $n$.  
Since both 
probabilities are increasing in $y$, this ratio also holds uniformly 
for $y'\in [y,y+M]$ and fixed $M>0$.  A similar
observation holds, as $n$ increases, for $\beta_{y, a, b}/y(y+a)$
with large (but bounded) values of $y$ and $a$, if $b-a$ is fixed. 
Note that the limits $\beta_{y,a,b}$ and $\beta_{y,a,b}^o$ may
depend on the sign of $d^{(n)}$, although $\beta^*$ will not.

\subsection{Preliminary bounds on the right tail of the maximum of BRW}

In this subsection, we give preliminary upper 
(Corollary \ref{lem-prelim-tail})  and lower (Lemma \ref{lem2.7}) 
bounds on the right tail of the maximum of BRW.  We first introduce
some terminology.

Throughout the paper, we will write $\{\eta_{v,n}(k)\}_ {k=0,1,\ldots,n}$ 
for the random walk where $\eta_{v,n}(k)$ is
the position of the $k$th generation individual in the 
family tree of individuals leading to
$v\in V_n$;
recall that $\eta_{v,n}(k+1)-\eta_{v,n}(k)$, $k=0,\ldots,n-1$, 
each have law $w(\cdot)$.
We also set $\bar{\eta}_{v,n}(k) = \eta_{v,n}(k) - km_n/n$.

For $\beta>0$, set
\begin{align}\label{eq-def-G-N-prelim}
G_{n,\beta} &= \bigcup_{v\in V_n}\bigcup_{0 \leq k \leq n}\{\bar{\eta}_{v, n}(k) \geq \beta + (4/\bar{\theta}) (\log(k \wedge (n-k)))_+\}\,,
\end{align}
where $\bar{\theta}$ is defined below (\ref{eq1.2}).
We also set
$g_{n,\delta}(i) = \exp\{-\delta |i|(\frac{|i|}{n \log n}\wedge1)\}$, where $\delta > 0$
is a constant that will be specified shortly.  

In order to show Corollary \ref{lem-prelim-tail}, we first obtain,
in Lemma \ref{lem-a-priori}, an upper bound 
on the probability that BRW takes atypically large values 
over $[0,n]$.  Lemma \ref{lem-a-priori} will also be applied in 
Section \ref{sec-limittail}.
\begin{lemma}\label{lem-a-priori}
There exists a constant $\delta >0$ such that
$\P(G_{n,\beta})  \lesssim \beta \mathrm{e}^{-\bar{\theta} \beta}g_{n,\delta}(\beta)$ for all $n\ge 2$ and $\beta \ge 1$.
\end{lemma}
For many of the applications in the next section, the weaker 
bound $\P(G_{n}(\beta))  \lesssim \beta \mathrm{e}^{-\bar{\theta} \beta}$ will suffice.  We remark that one can show the 
bound in Lemma \ref{lem-a-priori} 
still holds if the denominator $n\log n$ in $g_{n,\delta}(\cdot)$ is
replaced by $n$ (by using the Skorokhod embedding), 
although we have not done so here.

\begin{proof} [Proof of Lemma \ref{lem-a-priori}]
For given $v\in V_n$, we define the 
probability measure $\mathbb{Q}^{(n)}$, on paths in $[0,n]$, by
\begin{equation}
\label{eq-change-of-measure-1}
\frac{d\mathbb{P}}{d\mathbb{Q}^{(n)}} := 
\mathrm{e}^{-\theta_n \bar \eta_{v,n}(n) - nI(m_n/n)}
= (1 + O(\frac{1}{n}\log^2n))\,
n^{3/2}\rho^{-n}\mathrm{e}^{-\theta_n \bar \eta_{v,n}(n)} \,,
\end{equation}
where $I(\lambda)$ is the rate function in (\ref{eq1.1newer}) and 
$\theta_n = \theta_n (m_n/n)$ is the value of
$\theta$ at which the supremum in (\ref{eq1.1newer}) is
taken when $\lambda = m_n/n$.  The second equality
is a consequence of the definition of $m_n$ and the
differentiability of $I (\cdot)$, which imply that 
\begin{equation}
\label{eq-postPQ}
0 \le \bar{\theta} - \theta_n \lesssim \frac{1}{n}\log n \,,
\end{equation}
and of $I'(c_1) =\bar{\theta}$. 

For $0\leq k\leq n$, write $\psi_{n, \beta}(k) = \beta+
(4/\bar{\theta}) (\log (k \wedge (n-k)))_+$ and set
\begin{equation*}
\begin{split}
\chi_{n, j}^{\mathbb{P}}(i) &= \mathbb{P}(\bar \eta_{v,n}(k) \leq \psi_{n, \beta}(k) \mbox{ for all } k\leq j, \,\bar \eta_{v,n}(j) \in [i-1,i) ) \,,\\
\chi_{n, j}^{\mathbb{Q}^{(n)}}(i) &= \mathbb{Q}^{(n)}(\bar \eta_{v,n}(k) \leq \psi_{n, \beta}(k) \mbox{ for all } k\leq j, \,\bar \eta_{v,n}(j) \in [i-1,i) )\,.
\end{split}
\end{equation*}
By an elementary union bound,
\begin{align}
\label{eq-G-N-v}
\P(G_{n,\beta}) &\leq \sum_{j=0}^{n-1} \rho^{j+1}  
\sum_{i=-\infty}^{\lfloor\psi_{n,\beta}(j)+1\rfloor} 
\chi_{n, j}^{\mathbb{P}}(i) \P( i + \bar \eta_{v,n}(j+1) - 
\bar \eta_{v,n}(j) \geq \psi_{n, \beta}(j+1))\,.
\end{align}
We will obtain upper bounds for each of the two factors inside the inner sum
in  (\ref{eq-G-N-v}); the bound in (\ref{eq-nugen}) for the first factor
$\chi_{n, j}^{\mathbb{P}}(i)$ requires most of the work.   

For $\chi_{n, j}^{\mathbb{P}}(i)$,
we will need an upper bound on
$\chi_{n, j}^{\mathbb{Q}^{(n)}}(i)$, for which we consider the probability
measure $\mathbb{Q}$ on paths in $[0,n]$ defined by
$$\frac{d\mathbb{P}}{d\mathbb{Q}} = 
\mathrm{e}^{-\bar{\theta} (\eta_{v,n}(n) - c_1 n) - nI(c_1)},$$
for given $v\in V_n$.
Note that, under $\mathbb{Q}$, 
$\{\eta_{v,n}(k) - c_1 k\}_{k=0,1,\dots,n}$ is a mean zero random walk
that satisfies the moment assumptions of Lemma \ref{lemmaABR}, and, 
under $\mathbb{Q}^{(n)}$, 
$\{\bar{\eta}_{v,n}(k)\}_{k=0,1\ldots,n}$ is also a mean zero random walk.
Setting $S_k = \eta_{v,n}(k) - c_1 k$ and $S_k^{(n)} = \bar{\eta}_{v,n}(k)$,
one has $S_k^{(n)} = S_k +c_2 k (\log n)/n$, and $\mathbb{Q}$ and
$\mathbb{Q}^{(n)}$ satisfy the analog of (\ref{eq-change-of-measure-0})
for $\theta^{(n)}$ chosen as in the lemma.
The assumptions of Lemma \ref{prop-1DRW} are therefore satisfied
for $S_k$ and $S^{(n)}_k$, and 
consequently, by \eqref{eq-S-y-new} of Lemma 
\ref{prop-1DRW},
\begin{equation}
\label{eqnu0}
\chi_{n, j}^{\mathbb{Q}^{(n)}}(i) \lesssim j^{-3/2} 
\psi_{n, \beta}(0) (\psi_{n, \beta}(j) - i+2)
\end{equation}
for $i\leq \psi_{n, \beta}(j) +1$ and $n/2\le j \le n$.

Since $\chi_{n, j}^{\mathbb{P}}(i) \lesssim 
\frac{d \P}{d \mathbb{Q}^{(n)}}|_j ([i-1, i))\chi_{n, j}^{\mathbb{Q}^{(n)}}(i)$,
(\ref{eqnu0}), together with \eqref{eq-change-of-measure-1},
(\ref{eq-postPQ}),
and $\frac{1}{n}\log n \le \frac{1}{j}\log j$, implies that, for given $C>0$, 
\begin{equation}\label{eq-nu}
\chi_{n, j}^{\mathbb{P}}(i) \lesssim 
\psi_{n, \beta}(0) (\psi_{n, \beta}(j) - i+2) \rho^{-j} \mathrm{e}^{-\theta_n i}   
\lesssim  \beta (\psi_{n, \beta}(j) - i+2)\rho^{-j} \mathrm{e}^{-\bar{\theta} i} 
\end{equation}
for $\beta -C \sqrt{n \log n} \le i \le (\psi_{n,\beta}(j) +1) \wedge (n/\log n)$ and
$n/2\le j \le n$.
When $0\le j < n/2$, instead of $n/2\le j \le n$, is assumed
and the same range of $i$ is kept, one obtains 
 from \eqref{eq-change-of-measure-1} the
simpler
\begin{equation}
\label{eqnulowj}
\chi_{n, j}^{\mathbb{P}}(i) \le \P(\bar \eta_{v,n}(j) \in [i-1,i) )\lesssim 
j \rho^{-j} \mathrm{e}^{-\bar{\theta} i}\, ; 
\end{equation}
we will denote this last collection of pairs $(i,j)$ by $A_n$.
(Later on, the term $\rho^{-j}$ in (\ref{eq-nu}) and \eqref{eqnulowj} 
will cancel with the corresponding 
prefactor in (\ref{eq-G-N-v}), and $\mathrm{e}^{-\bar{\theta} i}$ will cancel with the
corresponding term in (\ref{eq-to-add-details-1}).)
The bound on the right hand sides of (\ref{eq-nu})  and \eqref{eqnulowj}
still holds after
multiplication by $g_{n,\delta}(i)$ on the right (because of the above
lower bound on $i$).
On the other hand, since the distribution of 
$w(\cdot)$ has exponential moments in a neighborhood
of $\bar{\theta}$, it will follow from
a moderate deviation estimate (using Markov's inequality) 
and (\ref{eq-postPQ}) that, for both $i < \beta -C \sqrt{n \log n}$
and $n/\log n < i \le \psi_{n,\beta}(j) +1$,
\begin{equation}
\label{eq23here}
\chi_{n, j}^{\mathbb{P}}(i) \le \P(\bar \eta_{v,n}(j) \in [i-1,i) )\lesssim 
j \rho^{-j} \mathrm{e}^{-\theta_n i} \exp\{-\delta i(\frac{i}{n}\wedge 1)\} \le 
 \rho^{-j} \mathrm{e}^{-\bar{\theta} i} g_{n,\delta}(i)
\end{equation}
for small enough $\delta >0$ and large enough $C$. 
(See, e.g., Dembo and Zeitouni \cite[Theorem 3.7.1]{DZ} for the moderate deviation
estimate.)  Grouping  (\ref{eq-nu}) and (\ref{eq23here}) together, one obtains
\begin{equation}
\label{eq-nugen}
\chi_{n, j}^{\mathbb{P}}(i) \lesssim \begin{cases}
\beta (\psi_{n, \beta}(j) - i+2)\rho^{-j} 
\mathrm{e}^{-\bar{\theta} i} g_{n,\delta}(i) & \mbox{for } (i,j)  \notin A_n  \,, \\
j\rho^{-j} \mathrm{e}^{-\bar{\theta} i} g_{n,\delta}(i) & \mbox{for } (i,j)  \in A_n  \,.
\end{cases}
\end{equation}

For the upper bound of the second factor in (\ref{eq-G-N-v}), note that,
since $w(\cdot)$ has exponential moments 
in a neighborhood of $\bar{\theta}$,
\begin{align}\label{eq-to-add-details-1}
\P(i + \bar \eta_{v,N}(j+1) - \bar\eta_{v,N}(j)
  \geq \psi_{n, \beta}(j+1))\lesssim 
\exp\{-(\bar{\theta}+\delta')(\psi_{n,\beta}(j+1) - i)\}\,
\end{align}
for some $\delta'>0$.

Plugging (\ref{eq-nugen}) and (\ref{eq-to-add-details-1}) 
into (\ref{eq-G-N-v}) 
and summing over the inner sum implies that
\begin{equation}
\label{eq25here}
\P(G_{n,\beta}) \lesssim \sum_{j=0}^{n} \beta ( (j+1) \wedge (n+1-j))^{-2}
 \mathrm{e}^{-\bar{\theta}\beta} g_{n,\delta}(\beta) 
\lesssim \beta \mathrm{e}^{-\bar{\theta} \beta} g_{n,\delta}(\beta) 
\end{equation}
for small enough $\delta >0$ and all $\beta \ge 1$,
where the power $-2$ in $((j+1) \wedge (n+1-j))^{-2}$ 
is obtained from the term
$(4/\bar{\theta}) (\log (j \wedge (n-j)))_+$.  
This completes the proof of the lemma.
\end{proof}

Our main application of Lemma \ref{lem-a-priori} in this section is the following
upper bound on $\P(\eta_n^*> m_n + z )$.  Let 
$\theta^*_n:=  \bar{\theta}$ for $z \le n$ and 
$\theta^*_n:=  \bar{\theta} + \delta$ for  $z > n$.  The upper bound
involving $\theta_n^*$, in Corollary \ref{lem-prelim-tail}, 
will suffice except in two places
((\ref{eqellbd}) and (\ref{eq4.51new}) of Lemma \ref{lem-second-moment}).
\begin{cor}\label{lem-prelim-tail}
For appropriate $\delta > 0$ and all  $n,z\geq 2$,
\begin{equation}\label{eq-large-tail}
\P(\eta_n^*> m_n + z ) 
\lesssim z \mathrm{e}^{-\bar{\theta} z}g_{n,\delta}(z)\,.
\end{equation}
In particular, 
$\P(\eta_n^*> m_n + z ) \lesssim z \mathrm{e}^{-\theta^*_n z}$
for all  $n,z\geq 2$.
\end{cor}
\begin{proof}
Since $\{\eta_n^*> m_n + z\} \subseteq G_{n,z}$, the bound in (\ref{eq-large-tail})
follows immediately from Lemma~\ref{lem-a-priori}.
\end{proof}

The following result is a quick consequence of 
Corollary \ref{lem-prelim-tail} and the definition of $m_n$.
\begin{cor}
\label{cor-new}
For appropriate $\delta > 0$, and all
$2\le \ell \le \sqrt{n}$ and $z\ge -\log \ell + 1$,
\begin{equation}\label{eq-large2}
\P(\eta_{\ell}^* > \ell m_n/n + z ) \lesssim \ell^{-3/2}(z+\log \ell) \mathrm{e}^{-\bar{\theta}z} g_{\ell,\delta}(z)\,.
\end{equation}
In particular,
$\P(\eta_{\ell}^* > \ell m_n/n + z ) \lesssim \ell^{-3/2}(z+\log \ell) \mathrm{e}^{-\theta^*_{\ell} z}$.
\end{cor}

The last result of the section gives a lower bound 
on the right tail of the maximum of BRW.
For $v, w \in V_n$, we say that $v$ and $w$ \emph{split}
at time $j_s = n-s$, denoted by $v\sim_s w$,
if $s$ is the maximal integer
such that $\{\eta_{v,n}(j) - \eta_{v,n}(j_s): j_s\leq j\leq n\}$
is independent of $\{\eta_{w,n}(j) - \eta_{w,n}(j_s): j_s\leq j\leq n\}$,
i.e., the last common ancestor of $v$ and $w$ occurs at time $j_s$.
\begin{lemma}
\label{lem2.7}
For all $n$ and $z$ satisfying $z\le \sqrt{n}$, 
\begin{equation}
\label{eq2.7}
\P(\eta_n^* > m_n + z ) \gtrsim z\mathrm{e}^{-\bar{\theta} z}\,.
\end{equation}
\end{lemma}
The argument for Lemma \ref{lem2.7}
involves well-known second moment estimates.  
More precise second moment estimates will be 
shown in Proposition \ref{prop-limiting-tail-gff}.
\begin{proof}[Proof of Lemma \ref{lem2.7}]
It suffices to show (\ref{eq2.7}) over $1\le z\le \sqrt{n}$.  For $v\in V_n$,
set 
\begin{equation*}
H_{v,n}(z) = \{\bar{\eta}_{v,n}(k) \le z
\text{ for all } k\le n-1, \, \bar{\eta}_{v,n}(n) \in (z, z+1]\}
\end{equation*}
and $\Delta_{n,z} = \sum_{v\in V_n} \one_{H_{v,n}(z) }$.
We will apply the elementary bound
\begin{equation}
\label{eq2.8}
\P (\eta^*_n > m_n +z) \ge (\E\Delta_{n,z})^2/\E(\Delta_{n,z})^2\,,
\end{equation}
which is a consequence of Jensen's inequality.

We obtain a lower bound on $\E\Delta_{n,z}$ by employing the change of measure in
(\ref{eq-change-of-measure-1}) as was done immediately below (\ref{eqnu0}),
but reversing the inequalities there and applying (\ref{equationABR2}) instead of 
(\ref{eq-S-y-new})
for the lower bound corresponding to (\ref{eqnu0}) 
(and with $z$ in place of $\psi_{n,\beta}(\cdot)$).
Multiplying this bound by $\rho^n$, we obtain
\begin{equation}
\label{eq2.9}
\E\Delta_{n,z} \gtrsim z\mathrm{e}^{-\theta_nz} 
\ge z\mathrm{e}^{-\bar{\theta} z} \,,
\end{equation}
with the first inequality holding for $1\le z \le \sqrt{n}$. 

For the upper bound on $\E(\Delta_{n,z})^2$, we employ the
decomposition
\begin{equation}
\label{eq2.10}
\E(\Delta_{n, z})^2 = 
\E\Delta_{n, z}  + 
K^*\rho^{n-2}\sum_{s=1}^n \rho^s\,\P(H_{v, n}(z) \cap H_{w, n}(z)
\text{ for } v\sim_s w )\,,
\end{equation}
where $K^* =K-\rho$, and $K$ is defined in (\ref{eq1.1}).
Set $J_i = z+ (-i-1,-i]$. 
Conditioning on the value at $\bar{\eta}_{v,n}(n-s)$,
one has, for $v\sim_s w$,
\begin{equation*}
\begin{split}
&\P(H_{v, n}(z) \cap H_{w, n}(z))\le
\sum_{i=0}^{\infty} 
\P(\bar{\eta}_{v,n}(k) \le z \text{ for } k < n-s; \, 
\bar{\eta}_{v,n}(n-s)\in J_i) \\
&\times  (\sup_{y\in J_i}
(\P (\bar{\eta}_{v,n}(n-s+j)\le z \text{ for } j< s;\,\bar{\eta}_{v,n}(n)\in 
(z,z+1]) \,|\, \bar{\eta}_{v,n}(n-s)=y))^2\,.
\end{split}
\end{equation*}
By the same reasoning as in  
(\ref{eq-change-of-measure-1})--(\ref{eq-nu})
for $\chi^{\mathbb{P}}_{n,j}(\cdot)$, one obtains upper
bounds (up to multiplicative constants) 
for the probabilities on the right hand side of 
the above display: two applications of (\ref{eq-S-y-new}) yield
the bound
$$ \frac{zi}{((n-s)\vee 1)^{3/2}}   \cdot (\frac{i}{s^{3/2}})^2  $$
and two applications of (\ref{eq-change-of-measure-1}) yield
$$\mathrm{e}^{-\theta_n(z+i)} \rho^{-(n+s)}n^{3/2}  
\exp\{((3/2)s\log n)/n\} \,,$$
with the summands in the above display being bounded 
by the product of these two quantities.
For $1 \le z \le\sqrt{n}$, substituting these bounds into the 
above display and factoring out
the terms not involving $i$ gives the sum, 
$\sum_{i=0}^{\infty} i^3 \mathrm{e}^{-\theta_n i} \le
\sum_{i=0}^{\infty} i^3 \mathrm{e}^{-\bar{\theta}i/2} < \infty$.
Consequently,
\begin{equation*}
\begin{split}
K^*\rho^{n-2}\sum_{s=1}^n \rho^s \, \P(H_{v, n}(z) \cap H_{w, n}(z)
\text{ for }v\sim_s w )
&\lesssim
z\mathrm{e}^{-\bar{\theta}z} \sum_{s=1}^n 
\frac{n^{3/2}\exp\{(3/2)(s/n)\log n\}} 
{((n-s)\vee 1)^{3/2}s^3} \\
&\lesssim z\mathrm{e}^{-\bar{\theta}z} \sum_{s=1}^{\infty} s^{-3/2}
\lesssim z\mathrm{e}^{-\bar{\theta}z} \lesssim \E\Delta_{n, z}\,,
\end{split} 
\end{equation*}
where the second inequality uses $(s/n)\log n \le \log s$ and
$(n-s)\vee s \ge n/2$, and the last inequality follows from (\ref{eq2.9}).
It therefore follows from (\ref{eq2.10}) that 
$\E(\Delta_{n, z})^2 \lesssim \E\Delta_{n, z}$.  This, together with
(\ref{eq2.8}), (\ref{eq2.9}) and (\ref{eq2.10}), implies (\ref{eq2.7}) and completes the proof 
of the lemma.
\end{proof}

\section{The limiting right tail of the maximum of BRW}
\label{sec-limittail}

The main result of this  section is the following proposition.  
\begin{prop}\label{prop-limiting-tail-gff}
There exists a constant  $\alpha^*>0$ such that
\begin{equation}
\label{eqforP3.1}
\lim_{z\to \infty}\limsup_{n\to \infty}|z^{-1} \mathrm{e}^{\bar{\theta}z}\P(\eta^*_n >  m_n +z) - \alpha^*| =0\,.
\end{equation}
\end{prop}

The section consists of two parts.  In Subsection 
\ref{firstsubsection}, our main
result is Proposition \ref{prop-gff-first-moment-dictates}, which 
compares $\P(\eta_n^* > m_n+z)$ with $E\Lambda_{n,z}$, which
is defined in (\ref{eq-big-definition}).  The main result in Subsection 
\ref{thesecondsubsection}, 
Proposition \ref{prop-asymptotic-first-moment}, shows that
$E\Lambda_{n,z} \sim \alpha^*ze^{-\bar{\theta}z}$ for large $z$.

\subsection{Expectation bounds for $\P(\eta_n^*> m_n+z)$}
\label{firstsubsection}
In order to prove Proposition~\ref{prop-limiting-tail-gff},
we will study the BRW at intermediate times $n-\ell\in (0,n)$,
where $\ell = \ell (z)$ is an integer function of $z$, not depending on $n$, and which satisfies
\begin{equation}\label{eq-L-tilde-L}
 \ell (z) \le z , \quad \ell (z)\rightarrow_{z\rightarrow\infty}\infty\,.
\end{equation}
When taking multiple limits, we will
let $n\rightarrow\infty$
before $z\rightarrow\infty$.  
(The restriction $\ell(z)\le z$ is employed, for example, 
in (\ref{eq3.48nnn}), but is needed because of
the term $\log n$ in $g_{n,\delta}(\cdot)$.)

For $v^{\prime}\in V_{n-\ell}$, define
\begin{equation}\label{eq-big-definition}
\begin{split}
E_{v^{\prime}, n}(z) &= \{\eta_{v^{\prime},n-\ell}(j) \leq jm_n/n + z  \mbox{ for all } 0\leq j\leq n-\ell, \mbox{ and } \max_{v\in V^{v^{\prime}}_{\ell}} \eta_{v,n} > m_n + z\}\,,\\
F_{v^{\prime}, n}(z) &= \{\eta_{v^{\prime},n-\ell}(j) \leq  jm_n/n +z  
+ \frac{1}{2}\log \ell + \frac{4}{\bar{\theta}} (\log [j \wedge (n-\ell-j)])_+ \\
&\qquad\mbox{ for all } 0\leq j\leq n-\ell, \mbox{ and } \max_{v\in V^{v^{\prime}}_{\ell}} \eta_{v,n} > m_n + z\}\,,\\
G_{n}(z) &= \bigcup_{v^{\prime}\in V_{n-\ell}}
\bigcup_{0\leq j\leq n-\ell}\{\eta_{v^{\prime},n-\ell}(j) >
jm_n/n +z+\frac{1}{2}\log \ell  + \frac{4}{\bar{\theta}} (\log [j \wedge (n-\ell -j)])_+ \}\,.\\
\end{split}
\end{equation}
Also define
$$\Lambda_{n, z} = \sum_{v^{\prime}\in V_{n-\ell}} \one_{E_{v^{\prime}, n}(z) }\,, \quad   \Gamma_{n, z} = \sum_{v^{\prime}\in V_{n-\ell}}
\one_{F_{v^{\prime}, n}(z) }\,.$$
In words, the random variable $\Lambda_{n,z}$ counts the number
of ($n-\ell$)th generation individuals $v^{\prime}$ for which
(i) over $j\in [0,n-\ell]$, $\eta_{v^{\prime},n}(\cdot)$,
stays below the line connecting $(0,z)$ to $(n,m_n + z)$ and (ii) at least
one of its descendents at time $n$ has position greater than
$m_n+z$.  The random variable $\Gamma_{n,z}$  
counts the number of individuals $v'$ whose ancestors
are instead constrained to stay below a higher, slightly concave curve.
(Here and later on, we will often suppress $\ell$ from the notation.)

The main result of this subsection is the following proposition.
\begin{prop}\label{prop-gff-first-moment-dictates}
For $\Lambda_{n,z}$ defined as above,
\begin{equation}
\label{eq-prop4.7}
\lim_{z\to \infty} \limsup_{n\to \infty} \frac{\P(\eta_n^* > m_n +z )}{\E \Lambda_{n, z}}  =  \lim_{z\to \infty} \liminf_{n\to \infty} \frac{\P(\eta_n^* > m_n +z )}{\E \Lambda_{n, z }} = 1\,.
\end{equation}
\end{prop}

In order to demonstrate Proposition \ref{prop-gff-first-moment-dictates},
we separately derive 
lower and upper bounds on 
truncations of the BRW at time $n-\ell$ in terms of the
curves in the
definitions of $\Lambda_{n,z}$ and $\Gamma_{ n,z}$.  
The following two requirements motivate our choices of $\Lambda_{n,z}$ and $\Gamma_{ n,z}$:
\begin{itemize}
  \item[(1)] The truncations corresponding to $\Lambda_{n,z}$ and 
$\Gamma_{n,z}$ should result asymptotically in the
same expectation; this will be shown
in Lemma~\ref{lem-Gamma-Lambda}.
\item[(2)]
  After truncation with respect to the curve corresponding to
$\Lambda_{n,z}$, the resulting second moment of the number of curves
  should be asymptotically the same as the corresponding expectation; this
  will be shown in Lemma~\ref{lem-second-moment}. 
  \end{itemize}

We first compare $\E\Lambda_{n, z}$ and
$\E\Gamma_{n, z}$.  Note that $\E\Lambda_{n, z} \le \E\Gamma_{n, z}$.
\begin{lemma}\label{lem-Gamma-Lambda}
For $\Lambda_{n, z}$ and $\Gamma_{n, z}$ as above,
\begin{equation}
  \label{eq-clear240113}
\lim_{z\to \infty} \liminf_{n\to \infty}
\frac{\E \Lambda_{n, z}}{\E \Gamma_{n, z}}
= 1\,.
\end{equation}
\end{lemma}
\begin{proof}
For $v'\in V_{n-\ell}$,
we write $\hat \eta_{v',n-\ell}(k) = \eta_{v',n-\ell}(k) - km_n/n$,
and define the probability measures $\mathbb{Q}^{(n)}$, on paths in $[0,n-\ell]$, by
\begin{equation}\label{eq-change-of-measure2}
\frac{d\mathbb{P}}{d\mathbb{Q}^{(n)}} = \mathrm{e}^{-\theta_n\hat \eta_{v',n}(n-\ell) - (n-\ell)I(m_n/n)}\,,
\end{equation}
where $\theta_n$ is defined below (\ref{eq-change-of-measure-1}).
Under $\mathbb{Q}^{(n)}$, $\hat \eta_{v',n-\ell}(\cdot)$ is a random walk
with 
increments whose law depends on $n$ but possesses a variance that is 
uniformly bounded in $n$.

For $\psi_{n,\ell,z}(k) = z + \frac{1}{2}\log \ell +
\frac{4}{\bar{\theta}}(\log[k\wedge (n-\ell -k)])_+$
and $J_i =z + \frac{1}{2}\log \ell + (-i-1,-i]$, set
$$\varphi^U_{n,\ell,z}(i) =\P(\hat{\eta}_{v',n-\ell}(k)\le \psi_{n,z}(k)
\text{ for all } k\le n-\ell, \,\hat{\eta}_{v',n-\ell}(n-\ell)\in J_i)\,,  $$
$$ \varphi^{L}_{n,\ell,z}(i) =\P(\hat{\eta}_{v',n-\ell}(k)\le z
\text{ for all } k\le n-\ell, \,\hat{\eta}_{v',n-\ell}(n-\ell)\in J_i)\,. $$
One then has the upper bound
\begin{align*}
\P(F_{v', n} (z)\setminus E_{v', n} (z))  
\le& \sum_{i\in A_1\cup A_3}
\varphi^U_{n,\ell,z}(i)
\P(\eta^*_{\ell} > \ell m_n/n + i - \frac{1}{2}\log \ell )  
\\
&+ \sum_{i\in A_2} (\varphi^U_{n,\ell,z}(i) - \varphi^L_{n,\ell,z}(i))
\P(\eta^*_{\ell} > \ell m_n/n + i - \frac{1}{2}\log \ell)\,,
\end{align*}
where $A_1$, $A_2$, and  $A_3$ are the integers restricted to
$[0,\,\ell^{1/3}+\frac{1}{2}\log \ell]$, 
$(\ell^{1/3}+\frac{1}{2}\log \ell\,,\ell +z]$, and $(\ell + z,\infty)$, respectively.  
One can bound the first sum over $A_1$, 
respectively, over $A_3$, by using the analog of (\ref{eq-nu}), respectively,
(\ref{eq23here}), to bound $\varphi^U_{n,\ell,z}(i)$, and by using
Corollary \ref{cor-new} to bound the second term, from which one 
obtains the upper bounds of the sums over $A_1$ and $A_3$,
$$C\rho^{-(n-\ell)} \ell^{-1/2} z\mathrm{e}^{-\bar{\theta} z} \quad \text{and} \quad
C\rho^{-(n-\ell)} z^{3} \mathrm{e}^{-(\bar{\theta} +\delta) z}\,, $$
for an appropriate constant $C$ and large enough $n$, 
where $\delta > 0$ is as in 
Corollary \ref{cor-new}.  It moreover 
follows from the comment after 
Lemma \ref{prop-1DRW} that the sum over $A_2$ is at most
$\delta_{\frac{1}{2}\ell^{1/3}} \P(E_{v', n}(z))$, where 
$\delta_y \rightarrow_{y\to \infty} 0$ and $\delta_y$ is as in Lemma \ref{prop-1DRW}.
Summation over $v'\in V_{n-\ell}$ therefore implies that
\begin{equation}\label{eq-difference-Gamma-Lambda}
\E \Gamma_{n, z}- \E \Lambda_{n, z} \lesssim  \ell^{-1/2}z \mathrm{e}^{-\bar{\theta}z} + \delta_{\frac{1}{2}\ell^{1/3}} \E \Lambda_{n, z}\,.
\end{equation}

By Lemmas \ref{lem-a-priori} and \ref{lem2.7} (applied with
$\beta=z+\frac{1}{2}\log \ell$), for all $\ell,n\ge 2$ and $z\le \sqrt{n}$,
\begin{equation}\label{eq-lower-Gamma}
\E \Gamma_{n, z} \gtrsim z\mathrm{e}^{-\bar{\theta} z}\,.
\end{equation} Together,
 \eqref{eq-difference-Gamma-Lambda}  and (\ref{eq-lower-Gamma}) imply (\ref{eq-clear240113}),
which completes the proof of the lemma.
\end{proof}

We next provide a precise estimate of  the second moment 
of $\Lambda_{n, z}$ in terms of its first moment.
\begin{lemma}\label{lem-second-moment}
For $\Lambda_{n,z}$ as above,
\begin{equation}
\label{newequation72}
\lim_{z\to \infty} \limsup_{n \to \infty} \frac{\E (\Lambda_{n, z})^2}{\E \Lambda_{n, z}} = 1\,.
\end{equation}
\end{lemma}
\begin{proof}
By
Lemma~\ref {lem-Gamma-Lambda}  and \eqref{eq-lower-Gamma}, 
\begin{equation}\label{eq-Lambda-first-moment}
\lim_{z\rightarrow\infty}\liminf_{n\rightarrow\infty} 
(\E \Lambda_{n, z} / z\mathrm{e}^{-\bar{\theta}z}) \gtrsim 1\,.
\end{equation}
We need to estimate the above second moment, which we rewrite as
\begin{equation}
\label{newequation73}
E(\Lambda_{n, z})^2 = 
E\Lambda_{n, z}  + 
\sum_{v,w\in V_{n-\ell},v\neq w} \P(E_{v, n}(z) \cap E_{w, n}(z))\,.
\end{equation}
In analogy with
the previous section, for $v,w\in V_{n-\ell}$, we will write $v\sim_s w$
when $v$ and $w$ split at time $j_s = n - \ell - s$.

As in (\ref{eq2.10}) of Lemma \ref{lem2.7}, 
we will show that, for large $z$, the sum in (\ref{newequation73}) 
is small in comparison with the first term on
the right hand side.
For $v\sim_s w$, with given $s$, we will employ the upper bound
\begin{align}\label{eq-P-A-v-w}
&\P(E_{v, n} (z) \cap E_{w, n}(z))\nonumber\\
= &\P(\bar \eta_{v,n}(j) , \bar \eta_{w,n}(j) \leq z \mbox{ for all }j\in [0,n-\ell]; \, \max_{u\in V^v_{\ell}} \bar{\eta}_{u,n},  
\max_{u\in V^w_{\ell}} \bar{\eta}_{u,n} > z)\nonumber\\
= &\sum_{i=0}^{\infty}\P(\bar \eta_{v,n}(j), \bar \eta_{w,n}(j) \leq z \mbox{ for all }j\in [0,n-\ell];\, \max_{u\in V^v_{\ell}} \bar{\eta}_{u,n} , \max_{u\in V^w_{\ell}} \bar{\eta}_{u,n} > z; \bar \eta_{v,n}(j_s) \in J_i)\nonumber\\
\leq &\sum_{i=0}^{\infty}\P(\bar \eta_{v,n}(j) \leq z  \mbox{ for all } j\in [0,j_s]; \,\bar \eta_{v,n}(j_s) \in J_i) (\Gamma_{v, i, z, s})^2 \,,
\end{align}
where, in the above sums, $J_i = z + (-i-1, -i]$, and 
$$\Gamma_{v, i, z, s} = \sup_{\bar \eta_{v,n}(j_s)\in J_i}
\P(\bar \eta_{v,n} (j)\leq z \mbox{ for all } j_s<j\leq n-\ell,   \max_{u\in V^v_{\ell}} \bar{\eta}_{u,n} > z
 \mid \bar \eta_{v,n}(j_s) ).$$

We decompose the range for $s$ into three regions, given by 
$[0,\lfloor \ell^{1/3}\rfloor)$, 
$[\lfloor \ell^{1/3}\rfloor, n-\ell - \lfloor \ell^{1/3}\rfloor)$ and 
$[n-\ell-\lfloor \ell^{1/3}\rfloor, n-\ell]$; 
the arguments for each part are similar,
with minor differences.   We first handle the main interval, consisting
of $s\in [\lfloor \ell^{1/3}\rfloor, n-\ell - \lfloor \ell^{1/3}\rfloor)$. 

Set $\gamma_{h,\ell} = \P(\eta^*_{\ell} >\ell m_n/n + h)$. 
Restarting the BRW at time $j_s$, then applying the change 
of measure (\ref{eq-change-of-measure2}), 
(\ref{equationABR1}) of the ballot theorem, and reasoning  
similarly to the upper bound for $\chi_{n,j}^{\P}(\cdot)$ in 
(\ref{eq-nu})--(\ref{eq-nugen}), one obtains
\begin{align}\label{eq-Gamma-y-z-s}
\Gamma_{v, i, z, s } \leq & \sum_{h=-1}^{\infty} \P(\bar \eta_{v,s+\ell}(j)\leq i+1 \mbox{ for all } j \in [0,s], \bar \eta_{v,s+\ell}(s) 
\in (i-h-1, i-h]) \gamma_{h,\ell}\nonumber \\
\lesssim & \sum_{h=-1}^{\infty} (h+2)(i+1)\mathrm{e}^{-\bar{\theta} (i-h)} g_{s,\delta}(i-h)
\mathrm{e}^{-sI(m_n/n)} \gamma_{h,\ell}/s^{3/2}
\end{align}
for appropriate $\delta >0$; similarly,
\begin{align}\label{eq-decompose-x}
\P(\bar \eta_{v,n}(j) \leq z \mbox{ for all } j\in [0,j_s]; \,\bar \eta_{v,n}(j_s) \in J_i) \lesssim \frac{z(i + 1)}{j_s^{3/2}}
\mathrm{e}^{-\bar{\theta}(z-i)}
\mathrm{e}^{-j_s I(m_n/n)}\,.
\end{align}
Substitution of
\eqref{eq-Gamma-y-z-s} and (\ref{eq-decompose-x}) into \eqref{eq-P-A-v-w} implies that
\begin{equation}
\label{eqother}
\P(E_{v, n} (z) \cap E_{w, n}(z)) 
\lesssim 
z \mathrm{e}^{-\bar{\theta}z}
[\mathrm{e}^{-(n-\ell +s)I(m_n/n)} /(s^3 j_s^{3/2}) ]
[\sum_{i=0}^{\infty} (i+1)^3\mathrm{e}^{-\bar{\theta} i}]
[\sum_{h=-1}^{\infty} (h+2)\mathrm{e}^{\bar{\theta} h } 
\gamma_{h,\ell}]^2 \,.
\end{equation}
(Here, we have used $g_{s,\delta}(i-h) \le 1$; the term $g_{s,\delta}(i-h)$ will be needed when applying
\eqref{eq-Gamma-y-z-s} for $s \in [1,\lfloor \ell^{1/3}\rfloor)$.)   
Using the definitions of $c_1$ and
$m_n$, one has the upper bound, for the quantity in the first brackets,  
\begin{equation}
\label{eqsbd}
\rho^{-(n-\ell +s)} \frac{n^{3/2}}{(sj_s)^{3/2}}
\frac{\exp\{(3/2)(s/n) \log n \}}{s^{3/2}} 
\lesssim \rho^{-(n-\ell +s)} \frac{n^{3/2}}{(sj_s)^{3/2}} \,,
\end{equation}
whereas, plugging in Corollary \ref{cor-new}, one has, for
$n\ge \ell^2$, the upper bound for 
the quantity in the last brackets,
\begin{equation}
\label{eqellbd}
\ell^{-3/2}\sum_{h=-1}^{\infty} (h+2)(h+\log \ell)g_{\ell,\delta}(h)
\lesssim (\log \ell)^{3/2} \,,
\end{equation}
since the sum is of order $\sqrt{\ell \log \ell}$ times the variance 
$\ell \log \ell$ of
the corresponding normal.  
Together with (\ref{eqother}), these two bounds imply that, for 
$s\in [\lfloor \ell^{1/3}\rfloor, n-\ell - \lfloor \ell^{1/3}\rfloor)$,
\begin{equation}
\label{eq-conclusion}
\P(E_{v, n} (z) \cap E_{w, n}(z)) \lesssim 
z \mathrm{e}^{-\bar{\theta}z}\rho^{-(n-\ell +s)}
 \frac{n^{3/2}(\log \ell)^{3}}{(sj_s)^{3/2}}\,.
\end{equation}

The argument for $s\in [n-\ell-\lfloor \ell^{1/3}\rfloor, n-\ell]$ 
is essentially the
same as the previous argument, but, 
instead of (\ref{eq-decompose-x}), which employs the ballot
theorem, we use the simpler upper bound 
\begin{equation}
\label{eq-decompose-b}
\P(\bar \eta_{v,n}(j_s) \in J_i) \lesssim \frac{1}{(j_s \vee 1)^{1/2}}
\mathrm{e}^{-\bar{\theta}(z-i)}
\mathrm{e}^{-j_s I(m_n/n)},
\end{equation}
which avoids the coefficient $z$ in the numerator.
Continuing as above, instead of (\ref{eq-conclusion}), one obtains
that, for 
$s\in [n-\ell - \lfloor \ell^{1/3}\rfloor,n-\ell]$,
\begin{equation}
\label{eq-conclusion2}
\P(E_{v, n} (z) \cap E_{w, n}(z)) \lesssim 
\mathrm{e}^{-\bar{\theta}z}\rho^{-(n-\ell +s)}
 \frac{n^{3/2}(\log \ell)^{3}}{s^{3/2}(j_s \vee 1)^{1/2}}\,.
\end{equation}

The argument for $s\in [0,\lfloor \ell^{1/3}\rfloor )$ is also similar to that
for $s\in [\lfloor \ell^{1/3}\rfloor, n-\ell - \lfloor \ell^{1/3}\rfloor)$, but 
one retains the term $g_{s,\delta}(i-h)$ in (\ref{eq-Gamma-y-z-s}), and therefore replaces the double sum
in (\ref{eqother}) by 
\begin{equation}
\label{eq4.51new}
\sum_{i=0}^{\infty} (i+1)^3\mathrm{e}^{-\bar{\theta} i}
[\sum_{h=-1}^{\infty} (h+2)\mathrm{e}^{\bar{\theta} h } g_{s,\delta}(i-h)
\gamma_{h,\ell}]^2 = 
\sum_{i=0}^{\infty} (i+1)^3\mathrm{e}^{-\bar{\theta} i/3}
[\sum_{h=-1}^{\infty} (h+2)\mathrm{e}^{\bar{\theta} h } \mathrm{e}^{-\bar{\theta}i/3} g_{s,\delta}(i-h)
\gamma_{h,\ell}]^2\,.
\end{equation}
Bounding $\gamma_{h,\ell}$ as before, this is
\begin{equation}
\label{eq4.51}
\lesssim \sum_{i=0}^{\infty} (i+1)^3\mathrm{e}^{-\bar{\theta} i/3}
[\sum_{h=-1}^{\infty} \ell^{-3/2} (h+2) (h + \log \ell)\mathrm{e}^{-\bar{\theta}i/3} g_{s,\delta}(i-h)]^2 \,.
\end{equation}
By completing the square, one can show that, for appropriate 
$\epsilon = \epsilon_{\bar{\theta},\delta} > 0$,
$$ \mathrm{e}^{-\bar{\theta}i/3} g_{s,\delta}(i-h) \lesssim g_{s,\epsilon}(h) $$ 
for all $s$, $h$, and $i$.   Employing (\ref{eqellbd}), 
but with $g_{s,\epsilon}(h)$ in place of $g_{\ell, \delta}(h)$, it follows that (\ref{eq4.51}), 
and hence (\ref{eq4.51new}),  is at most (up to a constant multiple)
\begin{equation*}
\{\ell^{-3/2}[(s\log s)^{3/2} +(s\log s) \log \ell]\}^2 \lesssim \ell^{-1/2}\,,
\end{equation*}
on accout of $s\le \ell^{1/3}$.
Plugging in this bound for the product of the last two bracketed 
quantities in (\ref{eqother}), and employing (\ref{eqsbd}) for the 
first bracketed quantity, this implies that
\begin{equation}
\label{eq3.55}
\P(E_{v, n} (z) \cap E_{w, n}(z)) \lesssim 
z\mathrm{e}^{-\bar{\theta}z} \rho^{-(n-\ell +s)} \frac{n^{3/2}\ell^{-1/2}}
{((s \vee 1)j_s)^{3/2}}\,. 
\end{equation}

Using $\E (\sum_{v,w \in V_{n-\ell},v\neq w} \one_{v\sim_s w}) =
 K^*\rho^{n-\ell +s -2}$, the bounds in (\ref{eq-conclusion}), 
(\ref{eq-conclusion2}), and (\ref{eq3.55}) together show that
the sum on the right hand side of 
(\ref{newequation73}) is at most (up to a constant multiple)
\begin{equation}
\label{eq3.48}
 z \mathrm{e}^{-\bar{\theta}z}
\{\sum_{s=0}^{\lfloor \ell^{1/3}\rfloor} \frac{n^{3/2}\ell^{-1/2}}
{((s \vee 1)j_s)^{3/2}} 
+ \sum_{s=\lfloor \ell^{1/3}\rfloor+1}^{n-\ell -\lfloor \ell^{1/3}\rfloor}
 \frac{n^{3/2} (\log \ell)^{3}}{(sj_s)^{3/2}} 
+\sum_{s=n-\ell -\lfloor \ell^{1/3}\rfloor+1}^{n-\ell} \frac{n^{3/2}(\log \ell)^{3}\ell^{1/3}}{(s(j_s \vee 1))^{3/2}z}\}\,. 
\end{equation}
Since $\sum_{s=k}^{\infty} 1/s^{3/2} \lesssim 1/k^{1/2}$, this is,
for given $z$ and large $n$,
\begin{equation}
\label{eq3.48nnn}
\lesssim z \mathrm{e}^{-\bar{\theta}z}\{\ell^{-1/2} + 
(\log \ell)^{3}\ell^{-1/6} + 
(\log \ell)^{3}z^{-2/3} \} \lesssim \ell^{-1/8}z\mathrm{e}^{-\bar{\theta}z}
\lesssim \ell^{-1/8}\, \E \Lambda_{n,z} \,,  
\end{equation}
where the first two inequalities use (\ref{eq-L-tilde-L}) 
and the third inequality uses (\ref{eq-Lambda-first-moment}).
The coefficient $\ell^{-1/8}$ of
$\E\Lambda_{n,z}$  goes to $0$ as $z\rightarrow\infty$.  
This shows that, for large $z$, the sum in (\ref{newequation73}) is small in comparison with the preceding term in (\ref{newequation73}), which
 completes the proof of the lemma.
\end{proof}

We now complete the demonstration of 
Proposition \ref{prop-gff-first-moment-dictates}.

\begin{proof}[Proof of Proposition \ref{prop-gff-first-moment-dictates},]
  By a simpler version of the argument in Lemma~\ref{lem-a-priori}, 
\begin{equation}
\label{eq3.50nn}
\P(G_n(z)) \lesssim \ell^{-\bar{\theta}/2}z\mathrm{e}^{-\bar{\theta}z}
\end{equation}
for $\ell \le \sqrt{n}$;
the factor $\frac{1}{2}\log \ell$ in the definition of $G_n(z)$ 
has been employed here.
(The analog of (\ref{eq-nu}) (rather than (\ref{eq-nugen})) suffices, for
which one employs the change of measure (\ref{eq-change-of-measure2})
(rather than (\ref{eq-change-of-measure-1})).)
Together, (\ref{eq3.50nn}), 
Lemma \ref{lem-Gamma-Lambda}, (\ref{eq-lower-Gamma}), and the trivial
estimate
$$\P(G_{n}(z))+\E\Gamma_{n,z}\geq
\P(\eta_{n}^* > m_n+z)$$
imply the upper bound 
\begin{equation}
\label{eq3.60nn}
\limsup_{z\to \infty} \limsup_{n\to \infty}
\frac{\P(\eta_{n}^* > m_n +z )}
{\E \Lambda_{n, z}}\le 1 \,.
\end{equation}
On the other hand, the lower bound
\begin{equation}
\label{eq3.61nn}
\liminf_{z\to \infty} \liminf_{n\to \infty} 
\frac{\P(\eta_{n}^*> m_n + z)}
{\E \Lambda_{n, z}}\geq 1
\end{equation}
is an immediate consequence of  
Lemma~\ref{lem-second-moment} and the inequalities
\begin{align*}
&\P(\eta_{n}^*> m_n+z)
\geq
\P(\bigcup_{v' \in V_{n-\ell}} 
E_{v',n}(z))
\geq \frac{(\E\Lambda_{n,z})^2}
{\E(\Lambda_{n,z})^2}
\end{align*}
(with the latter following from Jensen's inequality).  Together,  (\ref{eq3.60nn}) and (\ref{eq3.61nn}) imply (\ref{eq-prop4.7}).
\end{proof}

\subsection{Asymptotics for $\E \Lambda_{n,z}$}
\label{thesecondsubsection}
This subsection is devoted to demonstrating Proposition \ref{prop-asymptotic-first-moment},
which gives the asymptotic behavior of $E\Lambda_{n, z}$ for large $n$ and $z$.
\begin{prop}\label{prop-asymptotic-first-moment}
There exists a constant $\alpha^*>0$ such that
\begin{equation}
\label{eqmainresultinsecondsection}
\lim_{z\to \infty}\limsup_{n\to \infty} 
\frac{\E \Lambda_{n, z}}{\alpha^*
z \mathrm{e}^{-\bar{\theta}z}} = \lim_{z\to \infty}\liminf_{n\to \infty} 
\frac{\E \Lambda_{n, z}}{\alpha^* z \mathrm{e}^{-\bar{\theta}z}} = 1\,.
\end{equation}
\end{prop}

Together, Propositions~\ref{prop-gff-first-moment-dictates}
and \ref{prop-asymptotic-first-moment} imply
Proposition~\ref{prop-limiting-tail-gff}.

For $v\in V_{n-\ell}$, denote by $\nu_{n,z}(\cdot)$ the measure satisfying
$$\nu_{n,z}(I) = \P(\bar{\eta}_{v,n}(j) \leq z  \mbox{ for all } 0\leq j\leq n-\ell;\, \bar{\eta}_{v,n}(n-\ell)  \in z+ I)$$
for all intervals $I \in \mathbb{R}$.
Also, set
$\gamma_{\ell}(y) = \P(\eta^*_{\ell} >\ell m_n/n + y)$.
From the definition of $E_{v,n}(z)$ in 
(\ref{eq-big-definition}) and $J_i := (-i-1,-i]$, $i=0,1,\ldots$, one has
\begin{equation}
\label{equpperbd}
\P(E_{v, n}(z); \, \bar{\eta}_{v,n}(n-\ell) \in z+I) 
= \int_{I\cap (-\infty,0]} \gamma_{\ell} (- y) \nu_{n,z}(dy) \le \sum_{i=0}^{\infty} \gamma_{\ell}(i)\nu_{n,z}(I\cap J_i)\,.
\end{equation}
We will denote by
$\Lambda_{n,z,I}$ the analog of $\Lambda_{n,z}$, but with 
the added restriction $\bar{\eta}_{v,n}(n-\ell) \in z+I$; then
$\E \Lambda_{n,z,I} = \rho^{n-\ell}
\P(E_{v, n}(z); \, \bar{\eta}_{v,n}(n-\ell) \in z+I)$, 
for any $v\in V_{n-\ell}$.

Set $L_{\ell} = (-\ell, -\ell^{2/5}]$; the following lemma 
shows that the main contribution to
$\E\Lambda_{n,z}$ is from values $y\in I := L_{\ell}$, as in (\ref{equpperbd}).
(The choice of the exponent $2/5$ here is somewhat arbitrary; only $0<2/5<1/2$ is used.)  In the lemma, we will treat $z$ and $\ell$ as
independent variables, and will only employ the relationship 
(\ref{eq-L-tilde-L}) at the end of the subsection.
\begin{lemma}\label{lem-Lambda-J-k}
For $\Lambda_{n,z,L_{\ell}}$ defined as above,
$$\lim_{z,\ell\to \infty} \liminf_{n\to \infty}  \frac{\E \Lambda_{n, z,L_\ell}}{\E \Lambda_{n, z}} = 1\,.$$
\end{lemma}
\begin{proof}
Using (\ref{equpperbd}), it suffices to show that 
\begin{equation}
\label{eqpreeqwas54}
\lim_{\ell\rightarrow\infty} \sup_{z\ge 1} \limsup_{n\rightarrow\infty} \rho^{n-\ell}{\Big(}\sum_{i=0}^{\infty} 
\gamma_{\ell}(i)\nu_{n,z}(I_h\cap J_i){\Big)}{\Big /} \E \Lambda_{n, z} = 0
\end{equation}
for $h=1,2$, with $I_1 = (-\infty, -\ell ]$ and $I_2 = (-\ell^{2/5},0]$.

As in previous applications, (\ref{equationABR1}) and the change of measure (\ref{eq-change-of-measure2}), together with
(\ref{eq-postPQ}),  imply that, for $z,\ell \le \sqrt{n}$, 
\begin{equation}
\label{eqwas54}
\nu_{n,z}(J_i) \lesssim (i+1)z e^{-\bar{\theta} (z-i)}\rho^{-(n-\ell)} 
\end{equation}
for $i\ge 0$.  Also, by Corollary \ref{cor-new},
\begin{equation}
\label{eqaftereqwas54}
\gamma_{\ell}(i) \lesssim (i+1 + \log \ell)\ell^{-3/2} 
\mathrm{e}^{-\theta_{\ell}^*i} \,.
\end{equation}
Combining (\ref{eqwas54}) and (\ref{eqaftereqwas54}), it follows that
$$ \sum_{i=\lfloor \ell \rfloor}^{\infty} \gamma_{\ell}(i)\nu_{n,z}( J_i) 
\lesssim \rho^{-(n-\ell)} z\mathrm{e}^{-\bar{\theta}z}
\sum_{i=\lfloor \ell \rfloor}^{\infty} 
(i+1 + \log \ell)^2\ell^{-3/2} \mathrm{e}^{-\delta i }
\lesssim  \mathrm{e}^{-\delta \ell/2 }
\rho^{-(n-\ell)}z\mathrm{e}^{-\bar{\theta}z}\,,$$
where $\delta >0$ is as in Corollary \ref{cor-new}, and
$$ \sum_{i=0}^{\lfloor \ell^{2/5} \rfloor} \gamma_{\ell}(i)\nu_{n,z}( J_i)
\lesssim \rho^{-(n-\ell)} 
z\mathrm{e}^{-\bar{\theta}z}\sum_{i=0}^{\lfloor \ell^{2/5} \rfloor} 
(i+1 + \log \ell)^2\ell^{-3/2}  
\lesssim \ell^{-3/10}\rho^{-(n-\ell)} z\mathrm{e}^{-\bar{\theta}z}\,. $$
These bounds, together with \eqref{eq-clear240113} and 
(\ref{eq-lower-Gamma}), imply (\ref{eqpreeqwas54}).
\end{proof}

We employ the previous lemma, together with (\ref{eq-S-z})
to demonstrate Proposition~\ref{prop-asymptotic-first-moment}.
\begin{proof}[Proof of Proposition~\ref{prop-asymptotic-first-moment}]
Write $x_n = \ell m_n/n - c_1\ell$, set 
$J_i^N = (-(i-1)/N +x_n,-i/N +x_n]$, $i=0,1,\ldots$, for given $N \in 
\mathbb{Z}_+$.  Note that, 
for fixed $\ell$, $x_n \rightarrow_{n\rightarrow\infty} 0$.
Similar reasoning to that leading to (\ref{eqwas54}), but with the
sharper (\ref{eq-S-z}) in place of 
(\ref{equationABR1}), implies that, for fixed $N$,
\begin{equation}
\label{eqwas54precise}
\lim_{z,i\rightarrow\infty}\underset{n\to \infty}
{\mathop{\underline{\overline{\lim }}}}\,
\rho^{n-\ell} \nu_{n,z}(J_i^N){\Big /} {\Big (} \beta^* ((i/N-x_n)/N)z 
e^{-\bar{\theta} (z-i/N+x_n)} {\Big )}
=1\,,
\end{equation}
where ${\mathop{\underline{\overline{\lim }}}}_{n\rightarrow\infty}f(n)$ 
is shorthand for the bounds given by both 
$\limsup_{n\rightarrow\infty} f(n)$
and $\liminf_{n\rightarrow\infty} f(n)$.
For the moment treating $z$ and $\ell$ as independent variables,  it
follows from (\ref{eqwas54precise}) that
$$\lim_{z,\ell\rightarrow\infty}\underset{n\to \infty }{\mathop{\underline{\overline{\lim }}}}\,
\rho^{n-\ell} \sum_{i/N\in -L_{\ell}}
\gamma_{\ell}(i/N)\nu_{n,z}(J_i^N) {\Big /}
{\Big (} \beta^*\sum_{i/N\in -L_{\ell}} 
\gamma_{\ell}(i/N) ((i/N-x_n)/N)z 
e^{-\bar{\theta} (z-i/N+x_n)} {\Big )}  = 1 $$
for fixed $N$.
Application of bounded convergence to the denominator, 
as $n\rightarrow\infty$, therefore implies
\begin{equation}
\label{eqcombequations}
\lim_{z,\ell\rightarrow\infty}\underset{n\to \infty }{\mathop{\underline{\overline{\lim }}}}\,
\rho^{n-\ell} \sum_{i/N\in -L_{\ell}}
\gamma_{\ell}(i/N)\nu_{n,z}(J_i^N) {\Big /}      
 {\Big (} \beta^*\sum_{i/N\in -L_{\ell}} \gamma_{\ell}(i/N) (i/N^2)z 
e^{-\bar{\theta} (z-i/N)}{\Big )}  =1\,.
\end{equation}

On the other hand, because of  
the monotonicity of $\gamma_{\ell}(\cdot)$ and
$\nu_{n,z}(J_i^N)$ on $-L_{\ell}$, for given $N$,
$$1 - 2\bar{\theta}/N \le
\lim_{z,\ell\rightarrow\infty}\underset{n\to \infty }{\mathop{\underline{\overline{\lim }}}}\
\E \Lambda_{n, z, L_\ell} 
{\Big /} {\Big (} \rho^{n-\ell} \sum_{i/N\in -L_{\ell}}
\gamma_{\ell}(i/N)\nu_{n,z}(J_i^N){\Big )}  \le  1 + 2\bar{\theta}/N \,, $$
where one uses
(\ref{eqwas54precise}) to bound
$\nu_{n,z}(J_{i-1}^N)/\nu_{n,z}(J_i^N)$.
Letting $N\rightarrow\infty$, it 
follows from this and (\ref{eqcombequations}) that
\begin{equation}
\label{eqmostlythru}
\lim_{z,\ell\rightarrow\infty}\underset{n\to \infty }{\mathop{\underline{\overline{\lim }}}}\
\E \Lambda_{n, z, L_\ell} {\Big /}
{\Big (} \beta^* ze^{-\bar{\theta} z}\int_{-L_{\ell}} y  
e^{\bar{\theta} y} \gamma_{\ell}(y)\,dy {\Big )}  =1 \,.
\end{equation}
It follows from Lemma \ref{lem-Lambda-J-k} that the analog of
(\ref{eqmostlythru}) also holds, with $\E \Lambda_{n, z}$ in 
place of $\E \Lambda_{n, z, L_\ell}$, and so  
\begin{equation}
\label{eqwillbe66}
\lim_{z,\ell\rightarrow\infty}\underset{n\to \infty }{\mathop{\underline{\overline{\lim }}}}\
\E \Lambda_{n, z} {\Big /}
{\Big (}  \beta^* ze^{-\bar{\theta} z}\int_{-L_{\ell}} y  
e^{\bar{\theta} y} \gamma_{\ell}(y)\,dy {\Big )}  =1 \,.
\end{equation}
(Using the same reasoning
as in Lemma \ref{lem-Lambda-J-k}, one could also replace the region of
integration $-L_{\ell}$ by $[0,\infty)$.)
%Consequently, for given $\varepsilon >0$ and $\ell \ge \ell_0$, 
%with $\ell_0$ large but fixed, the limiting ratio will lie 
%within $(1-\varepsilon, 1+ \varepsilon)$ as $z\rightarrow\infty$.  

By (\ref{eq-prop4.7}) of Proposition \ref{prop-gff-first-moment-dictates}, 
for $\ell = \ell (z)$ satisfying (\ref{eq-L-tilde-L}), 
$\P(\eta_n^* > m_n +z )/\E \Lambda_{n, z} \rightarrow 1$
as first $n\rightarrow\infty$ and then $z\rightarrow\infty$.
Since $\P(\eta_n^* > m_n +z )$ does not depend on the choice of  $\ell (z)$,
neither does $\E \Lambda_{n, z}=\E\Lambda_{n,z}^{\ell(z)}$ in the limit when
$\ell(z)$ is chosen according to \eqref{eq-L-tilde-L}. Explicitly,
if $\ell(z)$, $\ell'(z)$ both satisfy \eqref{eq-L-tilde-L}, then
$$\lim_{z\rightarrow\infty}
\underset{n\to \infty }{\mathop{\underline{\overline{\lim }}}}\
\frac{\E \Lambda_{n, z}^{\ell(z)}}
{\E \Lambda_{n, z}^{\ell'(z)}}=1\,.$$
Hence, because of 
(\ref{eqwillbe66}), 
$$\underset{\ell,\ell'\to \infty }{\lim }
\frac{  \int_{-L_{\ell}} y  e^{\bar{\theta} y} \gamma_{\ell}(y)\,dy }
{  \int_{-L_{\ell'}} y  e^{\bar{\theta} y} \gamma_{\ell'}(y)\,dy }
=1\,.$$
Therefore, 
 $\beta^*\! \int_{-L_{\ell}} y  e^{\bar{\theta} y} \gamma_{\ell}(y)\,dy $
must converge to a limit $\alpha^*$ as  $\ell\to\infty$.
%n\rightarrow\infty$ and then 
%$z\rightarrow\infty$, for any choice of $\ell (z)$ satisfying 
%(\ref{eq-L-tilde-L}).
On account of Lemma \ref{lem-Gamma-Lambda} and (\ref{eq-lower-Gamma}),
$\alpha^* > 0$.
Together with (\ref{eqwillbe66}), this implies the
limits in (\ref{eqmainresultinsecondsection}).
\end{proof}

\section{Proof of Theorem \ref{thm1.2}}
\label{sec-maintheorem}

We first show that, under the assumptions of Theorem \ref{thm1.2}, 
$\eta^*_n - m_n$ converges in distribution.  One obtains, 
by decomposing
the BRW over the 
time intervals $[0,k]$ and $[k,n]$,
\begin{equation}
\label{eq4.1.1}
\begin{split}
\P (\eta^*_n - m_n \le z\, |\, \mathcal{F}_k) &= 
\prod_{v' \in V_k} \P (\eta_{v',k}  + \tilde{\eta}^*_{n-k} 
- m_n \le + z) \\
&= \prod_{v' \in V_k} \P (\tilde{\eta}^*_{n-k} 
-  m_{n-k} \le m_n - m_{n-k} - \eta_{v',k} + z)\,,
\end{split}
\end{equation}
where $\eta_{v',k}$ are the positions of the particles for the process 
$\eta_.$ at time $k$, and $\tilde{\eta}_.$ is an independent BRW.
Note that, from the definition of $m_n$, 
$m_n - m_{n-k} = c_1k + \delta_k(n)$, with 
$\delta_k(n)\rightarrow_{n\rightarrow\infty}0$ for fixed $k$, and
that, by Proposition \ref{prop-limiting-tail-gff},  the
distribution of $\eta^*_{n-k} - m_{n-k}$ has right tail 
$\alpha^*z\mathrm{e}^{-\bar{\theta}z}$ 
for large $n-k$.  Since 
\begin{equation}
\label{eq4.1.1new}
\eta^*_k - c_1 k \rightarrow_{k\rightarrow\infty} -\infty
\end{equation}
in probability, it follows that, for large $k$ and much larger $n$,
the logarithm of the right hand side of (\ref{eq4.1.1}) is asymptotically

\begin{equation}
\label{eq4.1.2}
\sum_{v' \in V_k} \log [1 - \alpha ^* z_{v',k} \mathrm{e}^{-\bar{\theta}z_{v',k}} ]
\sim -\alpha^*\sum_{v'\in V_k} z_{v',k}\mathrm{e}^{-\bar{\theta}z_{v',k}},
\end{equation}
for fixed $z$,
where $z_{v',k} = c_1 k -\eta_{v',k} +z $. 
Setting
\begin{equation*}
Z_k = \sum_{v' \in V_k}(c_1k - \eta_{v',k})\mathrm{e}^{-\bar{\theta}(c_1k - \eta_{v',k})}\,, \quad 
Y_k = \sum_{v' \in V_k}\mathrm{e}^{-\bar{\theta}(c_1k - \eta_{v',k})}\,,
\end{equation*}
this equals  $-\alpha ^*\mathrm{e}^{-\bar{\theta}z}(Z_k + zY_k)
\sim -\alpha ^*\mathrm{e}^{-\bar{\theta}z}Z_k $, since
(\ref{eq4.1.1new})
implies that $Y_k/Z_k \rightarrow 0$ in probability.

It follows from the last paragraph that, for appropriate $A_k \in \mathcal{F}_k$ and 
$\epsilon_k > 0$, with $Z_k > 0$ on $A_k$ and
$\P(A^c_k) \le \epsilon_k \rightarrow_{k\rightarrow\infty} 0$, 
\begin{equation}
\label{eq4.2.1}
1-\epsilon_k \le 
\log \P (\eta^*_n - m_n \le z\, |\, \mathcal{F}_k) {\big /}
{\big (} -\alpha ^*\mathrm{e}^{-\bar{\theta}z}Z_k {\big )}
\le 1+\epsilon_k \quad \text{on } A_k,
\end{equation}
for $n\ge n_k$ and appropriate $n_k$, with 
$n_k \rightarrow_{k\rightarrow\infty}\infty$, and hence 
\begin{equation}
\label{eq4.2.2}
\E[\exp\{-(1+\epsilon_k)\alpha ^*\mathrm{e}^{-\bar{\theta}z}Z_k\};A_k] 
\le \P (\eta^*_n - m_n \le z) \le 
\E[\exp\{-(1-\epsilon_k)\alpha ^*\mathrm{e}^{-\bar{\theta}z}Z_k\};A_k] 
+ \epsilon_k\,.
\end{equation}
As $k\rightarrow\infty$, the difference of the left (and right) hand side 
of (\ref{eq4.2.2}) and   
$\E[\exp\{-\alpha ^*\mathrm{e}^{-\bar{\theta}z}Z_k\}; A_k]$ converges to $0$. 
(To see this, note that $\exp\{-\mathrm{e}^{-x}\}$ is uniformly 
continuous in $x$.)  Consequently, for fixed $z$,
\begin{equation}
\label{eq4.2.3}
\lim_{k\rightarrow\infty}\limsup_{n\rightarrow\infty}
|\P (\eta^*_n - m_n \le z) -
\E[\exp\{-\alpha ^*\mathrm{e}^{-\bar{\theta}z}Z_k\};A_k]| = 0\,.
\end{equation}
It follows from this that, for some function $w(\cdot)$,
\begin{equation}
\label{eq4.2.4}
\P (\eta^*_n - m_n \le z) \rightarrow_{n\rightarrow\infty} w(z) \,.
\end{equation}
One can check that $\lim_{z\rightarrow -\infty} w(z) = 0$ (since $\max_v\eta_{v,n}^{v'} -m_n$ are independent for 
different $v'\in V_1$, but the limits of their distributions
must share this same limit as first $n\rightarrow\infty$
and then $z\rightarrow -\infty$).
So $\eta^*_n - m_n $ is in fact tight,
and hence $w(\cdot)$ is a distribution function.
This completes the proof of the first part of Theorem \ref{thm1.2}.

We still need to verify (\ref{eq1.4}).  To show this, we will first show
that $\lim_{n\rightarrow}Z_n$ exists a.s.  (In probability convergence 
of $Z_n$ is automatic from (\ref{eq4.2.3}), although we do not use this.)
Because of (\ref{eq4.2.4}), 
$m_{n+1} - m_n \rightarrow_{n\rightarrow\infty} c_1$, and the 
BRW property,
\begin{equation*}
w(0) = \E[\prod_{v\in V_{1}}w(c_1-\eta_{v,1})].
\end{equation*}
Consequently, $W_{n}:=\prod_{v\in V_{n}}w(c_1n-\eta_{v,n})$
is a martingale with respect to $\mathcal{F}_{n}$.  Since $W_n$ is
nonnegative,  $W_n \rightarrow_{n\rightarrow\infty}W$ exists a.s.
by the martingale convergence theorem. On the other hand,
it follows from the definition of $c_1$ that 
\begin{equation*}
\E [\sum_{v\in V_1}\mathrm{e}^{-\bar{\theta}(c_1 - \eta_{v,1})}] = 1\,,
\end{equation*}
and hence $Y_n$ is also a nonnegative martingale.  Another application
of the martingale convergence theorem implies that 
$Y_n \rightarrow_{n\rightarrow\infty}Y$ exists a.s.  
(In fact, $Y=0$ since $Y_k/Z_k\to 0$ in probability.)
This implies,
in (\ref{eq4.1.1new}), the stronger a.s. convergence in fact holds.

It follows from Proposition \ref{prop-limiting-tail-gff} and (\ref{eq4.2.4})
that $1-w(z) \sim \alpha^*z\mathrm{e}^{-\bar{\theta}z}$ as
$z\rightarrow\infty$.  This same property was exploited for
$\P (\eta^*_n - m_n > z)\, |\, \mathcal{F}_k) $ in the first part of
the proof, which we now also use for $W_n$.  Proceeding 
as in (\ref{eq4.1.1})--(\ref{eq4.1.2}), but employing a.s. rather 
than in probability convergence, one obtains the following a.s. 
analog of (\ref{eq4.2.1}),
\begin{equation*}
(\log W) {\big /} \lim_{k\rightarrow\infty}(-\alpha ^*Z_k) = \lim_{k\rightarrow\infty}
(\log W_k) {\big /}(-\alpha ^*Z_k) = 1 \quad \text{a.s.}
\end{equation*}
Therefore, $Z = \lim_{k\rightarrow\infty}Z_k$ exists a.s., with
$Z = -(\alpha^*)^{-1}\log W$.

The reasoning leading up to (\ref{eq4.2.1}) also implies the a.s. version,
\begin{equation*}
\begin{split}
\lim_{k\rightarrow\infty}\liminf_{n\rightarrow\infty} 
&(\log \P (\eta^*_n - m_n \le z\, |\, \mathcal{F}_k)) {\big /}
(-\alpha ^*\mathrm{e}^{-\bar{\theta}z}Z_k) \\
&= \lim_{k\rightarrow\infty}\limsup_{n\rightarrow\infty} 
(\log \P (\eta^*_n - m_n \le z\, |\, \mathcal{F}_k)) {\big /}
(-\alpha ^*\mathrm{e}^{-\bar{\theta}z}Z_k)
=1 \quad \text{a.s. } 
\end{split}
\end{equation*}
Plugging in $Z = \lim_{k\rightarrow\infty}Z_k$ reduces the above limits
to (\ref{eq1.4}).  Taking expectations on both sides of \eqref{eq1.4},
and 
applying bounded convergence to the limit as $z\to -\infty$,
together with  
$w(-\infty) = 0$,
implies that $Z$ is a.s. strictly positive.  
This completes the proof of the second part of Theorem \ref{thm1.2}.

\section{The lattice case}
\label{sec-lattice}
The methods in this paper allow one to also handle the lattice case,
which is mentioned as an open problem in Aidekon \cite{Aidekon}. 
Unfortunately, in many places both the statements and some steps in the 
proof need to be modified, albeit in a minor way. To avoid repetitions
or burdening the main text with extra details geared toward the lattice case,
we decided to only summarize in this short section the result and the 
needed adaptations, and to leave the actual proof for either future work or
the interested reader.

Recall that a random walk with increments distribution
$w(\cdot)$ is called \textit{lattice} if, for some $y$,
the support
of $w(\cdot) + y$ 
is contained in 
a discrete subgroup of $\mathbb{R}$. By rescaling and shifting, we can
and will assume that $y=0$ and that the discrete subgroup is $\mathbb{Z}$.
%(The factor $4$ has no particular significance, and is used for convenience in 
%order to minimize the changes needed in Section \ref{sec-limittail}.)
%
We define $m_n$ and $Z_k$ as in the nonlattice case, and set 
$A_n=\mathbb{Z}-m_n$. 

The analog of Theorem \ref{thm1.2} is
the following.
\begin{theorem}
\label{thm1.2a}
Assume that $\eta_n$ is a lattice branching random walk
satisfying (\ref{eq1.2}), with $K < \infty$.  Then, 
$Z = \lim_{k\rightarrow\infty}Z_k$
exists and is finite and positive with probability $1$, and there exists a
constant $\alpha ^*>0$ so that
%, for each $z\in \mathbb{R}$,
\begin{equation}
\label{eq1.4lattice}
\lim_{k\rightarrow\infty}\lim_{n\rightarrow\infty}
\sup_{z\in A_n} |\P(\eta^*_n \le m_n +z |\,\mathcal{F}_k)
- \exp\{-\alpha ^*Z \mathrm{e}^{-\bar{\theta}z}\}|=0 \quad \text{a.s.}
\end{equation}
\end{theorem}

Note that, by taking expectations in (\ref{eq1.4lattice}), one obtains
\begin{equation}
\label{varianteq1.4lattice}
\lim_{n\to\infty} 
\sup_{z\in A_n} |\P(\eta^*_n \leq m_n+z)-\E
\exp\{-\alpha ^*Z \mathrm{e}^{-\bar{\theta}z}\}
|=0\,.
\end{equation}
If one wishes, one can also rephrase  (\ref{eq1.4lattice}) and (\ref{varianteq1.4lattice}) as convergence results over all $\mathbb{R}$ (rather than over $A_n$) by appropriate interpolation  of $\P(\cdot)$ within the intervals between lattice points.   

We indicate the main modifications in the argument for obtaining Theorem \ref{thm1.2a} in place of Theorem \ref{thm1.2}, while including 
some details for the curious reader. 
Most importantly, we need a modified version of the technical estimates 
in Subsection \ref{subsec-longlemma},  with the
 main modifications being in \eqref{eq-S-y}--(\ref{eq-S-y-gamma}) of Lemma 
\ref{prop-1DRW}.   In the lemma, we replace
the interval
$(a,b)$  by the 
half open interval $(a,b]$, with $b-a=1$, which
contains exactly one lattice point.  Also, we replace
the term $y$, $y\ge 1$, given in the definition of the 
boundary, by an appropriately chosen $y'$ with 
$|y'-y| < 1$.

In the proofs in the appendix of the different parts of 
Lemma \ref{prop-1DRW}, various details for the lattice case differ from
those for the nonlattice case.  
One needs to
replace the conditional local CLT  of \cite[Theorem 1]{Car} and the
asymptotic expansion of \cite[Theorem 16.4.1]{Fell} by their lattice
analogs \cite[Theorem 2]{Car} and \cite[Theorem 16.4.2]{Fell}.
The inequality (\ref{eqjg0.1}) obtained by reversing time 
can be replaced by a corresponding equality
if the interval $(a,b)$ is replaced by a single point; this
simplifies somewhat the proofs of (\ref{eq-S-y}) and (\ref{eq-S-z}).
In (\ref{eqjg1.1}) of the important Lemma \ref{lemjg}, one sets $\delta = 0$, since
expanding the interval by $\delta >0$ no longer ensures the inequality
(\ref{eqjg3.2}) needed for step (C) of the proof; rather, in 
(\ref{eqjg1.1}), one instead replaces the term $y$ by an appropriate $y'$
satisfying $|y'-y|<1$.  (The existence of such a $y'$ can be shown by
noting that
the maximum separation at any time between the two boundaries 
in (\ref{eqjg3.2}) is of order $(\log n)/n^{3/4}$, and considering the probability
that the random walk visits this region at some time, for   
$y'$ is chosen randomly in $[0,1)+y$.  Since the time interval
has length $\lfloor n^{1/4}\rfloor$, 
application of the union bound over these times shows that
the probability of this event occurring is
of smaller order than $1$.)

Once the estimates in Subsection \ref{subsec-longlemma} are obtained, the rest of the
proof in Sections \ref{sec-prelim}-- 
\ref{sec-maintheorem} proceeds for the most part as in the nonlattice case,  
with the change that,
in Section \ref{sec-limittail}, all limits in $z$ are taken over lattice points, which is reflected in the argument in Section \ref{sec-maintheorem}.
In particular,  (\ref{eqforP3.1}) of Proposition \ref{prop-limiting-tail-gff} and (\ref{eqmainresultinsecondsection}) of Proposition \ref{prop-asymptotic-first-moment} need to be modified in a manner analogous to (\ref{eq1.4lattice}) of Theorem \ref{thm1.2a}, with $z$ restricted to
$A_n$ at a given  time $n$.  (In the proof of Proposition 
\ref{prop-asymptotic-first-moment}  at the end
of the subsection, one also sets $N=1$ rather than letting $N\rightarrow\infty$.)  One also needs to modify slightly the definitions of the
terms $E_{v',n}(z)$, $F_{v',n}(z)$, and $G_n(z)$, 
given in (\ref{eq-big-definition}), because of the double role played
by $z$ there, which appears in both the definition of the boundary and
the value of the trajectory at time $n$; in the first instance, one
instead employs $z'$ with $|z'-z| < 1$.  (This modification does not affect
the bounds that follow because of the upper bound in (\ref{eq3.50nn}).)
The reasoning in Section \ref{sec-maintheorem} is analogous to that
for the nonlattice case, although one needs to put more effort
into constructing the function $w(\cdot)$ in (\ref{eq4.2.4}), because of 
the restraint $z\in A_n$, at time $n$, on the analog of the left hand side of (\ref{eq4.2.4}).  (To get around this, one can employ the limit  
$z_{n_i} \rightarrow_{i\rightarrow\infty} z$, for any fixed $z$, 
which will hold 
at appropriate times $n_i\rightarrow_{i\rightarrow \infty}
\infty$, and points $z_{n_i}$ with $z_{n_i} \in A_{n_i}$.)

\section*{Appendix A}
\renewcommand{\thesection}{\Alph{section}}
\setcounter{section}{1}
The appendix is devoted to the demonstration of 
Lemma \ref{lemmaABR} and Lemma  \ref{prop-1DRW}.  
The special case of Lemma \ref{lemmaABR} with $S_k^{(n)} = S_k$,
i.e., $d^{(n)} \equiv 0$,
was given in Lemma \ref{lemmaABRold}.  In the first subsection of
the appendix, we
demonstrate the special case of Lemma \ref{prop-1DRW} with 
$d^{(n)} \equiv 0$ and, for (\ref{eq-S-y-new}), when $j=n$ is also assumed.
In the second subsection, we then show, in Corollary \ref{corjg},
that essentially the same asymptotic behavior holds for $\{S_k\}_{k=0,\ldots,n}$
and $\{S_k^{(n)}\}_{k=0,\ldots,n}$ as $n\rightarrow\infty$, when
$d^{(n)} = O((\log n)/n)$.  
Lemmas \ref{lemmaABR} and \ref{prop-1DRW} will follow from
Lemma \ref{lemmaABRold},
the special case of Lemma \ref{lemmaABR}, and 
Corollary \ref{corjg}.

Variants of the following inequality will be used several times in 
the appendix.  Let $\tilde{S}_k$ be the random walk obtained
by reversing time, i.e., 
$\tilde{S}_k = \sum_{i=1}^k (-X_i)$, 
and define $\tilde{S}_k^{(n)}$ by 
$\tilde{S}_k^{(n)} = \sum_{i=1}^k -({X_i} + {d}^{(n)})$ for $k\le n$.
It is easy to check that, for a given choice of $a$, $b$, and $y$,
with $a<b$ and $y\ge 0$, 
\begin{equation}
\label{eqjg0.1}
\begin{split}
& \P(\tilde{S}_n^{(n)}\in (-b, -a), 
 \tilde{S}_k^{(n)}\geq -y- h(k\wedge (n-k)) -a \mbox { for all } 0 < k < n) \\
& \quad \le \P(S_n^{(n)}\in (a, b), 
S_k^{(n)}\geq -y- h(k\wedge (n-k)) \mbox { for all } 0 < k < n) \\
& \quad \le \P(\tilde{S}_n^{(n)}\in (-b, -a), 
\tilde{S}_k^{(n)}\geq -y- h(k\wedge (n-k)) -b \mbox { for all } 0 < k < n). 
\end{split}
\end{equation}
By first partioning $(-b,-a)$ into smaller intervals and then applying
the analog of (\ref{eqjg0.1}) to each subinterval, one obtains more
accurate bounds for the middle probability in (\ref{eqjg0.1}).

\subsection{Demonstration of Lemma \ref{prop-1DRW}, 
with $d^{(n)} = 0$}

Here, we demonstrate Lemma  \ref{prop-1DRW} with $d^{(n)} = 0$.  We
also restrict (\ref{eq-S-y-new}) to the case $j=n$.

For the first two parts of Lemma  \ref{prop-1DRW},
we will employ
the following conditioned local central limit theorem, which follows from
 [Caravenna \cite{Car}, Theorem 1].  As at the beginning of Section 2, 
$\{X_k\}_{k=1,2,\ldots}$ denotes independent copies of a mean zero
random variable $X$ and $S_n = \sum_{k=1}^n X_k$; $X$ is also 
assumed to be nonlattice.
We denote by $\mathcal{C}_n$ the set where $S_1,S_2,\ldots,S_n > 0$.

\begin{prop}\label{prop-car}
Let $X$ be as above, with variance $1$.  Then,
\begin{equation}
\label{eqA0.1}
\sup_{x\in \mathbb{R_+}} |n^{1/2}\, \P(S_n\in [x,x+q) |\,\mathcal{C}_n) 
- q x \mathrm{e}^{-x^2/2n}/n^{1/2}|\rightarrow_{ n\rightarrow\infty} 0 
\end{equation}
uniformly over  $q$ in compact sets in $\mathbb{R}_+$.
\end{prop}
Also note that, by [Caravenna \cite{Car}, (2.6)] (or from older references), 
\begin{equation}
\label{eqA0.2}
\P(\mathcal{C}_n) \sim  C/\sqrt{n} \quad \text{for some } C>0\,.
\end{equation}

In order to demonstrate (\ref{eq-S-y}) and (\ref{eq-S-z}) of
Lemma  \ref{prop-1DRW}, with $S_k^{(n)} = S_k$, 
we first demonstrate the limits (\ref{eqA0.3}) and (\ref{eqA0.5}) below,
which are extensions of (\ref{eqA0.1}), with  the 
restriction $S_1,S_2,\ldots,S_m\ge -y$,
for fixed $y\ge 1$, replacing that of 
$S_1,S_2,\ldots,S_m > 0$, for $m=\lfloor n/2 \rfloor$.
We then express $S_n$ by ``gluing together" 
independent copies of $S_m$ and
$-S_{n-m}$, and apply (\ref{eqA0.3}) and (\ref{eqA0.5}) to each
half to obtain (\ref{eq-S-y}) and (\ref{eq-S-z}).
\begin{proof}[Proof of (\ref{eq-S-y}) and (\ref{eq-S-z}) of
 Lemma \ref{prop-1DRW} for $S_k$] 
We may assume WLOG that 
$\sigma^2(X) = 1$ by rescaling the random walk.
In order to demonstrate (\ref{eq-S-y}) and (\ref{eq-S-z}),
we first show that, for appropriate $C>0$,
\begin{equation}
\label{eqA0.3}
\sup_{x\ge -y} |m\,\P(S_m\in [x,x+q); S_k \ge -y \text{ for all } 1\le k\le m)
-  \bar{\beta}_y q x \mathrm{e}^{-x^2/2m}/m^{1/2} |  
\rightarrow_{m\rightarrow\infty} 0\,,
\end{equation}
for fixed $y,q>0$, where $\bar{\beta}_y$ is right continuous and
monotone in $y$, and
\begin{equation}
\label{eqA0.5}
\bar{\beta}_y /y \rightarrow_{y\rightarrow\infty} \bar{\beta}^* >0\,.
\end{equation}
(Although the latter term in (\ref{eqA0.3}) is negative for $x<0$,
it will be negligible for large $m$ and will not affect the limit.)

To show (\ref{eqA0.3}) and (\ref{eqA0.5}), we introduce the following
terminology.  
Let $\tau$ be the first time at which $\{S_k: k\le n\}$ takes its
minimum.
Abbreviate the event on the left hand side of 
(\ref{eqA0.3}) by $\Psi_m^y([x,x+q))$, and partition
$\Psi_m^y([x,x+q))$ by 
$\{\Psi_{m,M}^{y,1}([x,x+q)),\Psi_{m,M}^{y,2}([x,x+q))\}$, where
$\Psi_{m,M}^{y,1}(\cdot)$ is given by the restriction
$\tau \le \lfloor M^2y^2\rfloor$ and $\Psi_{m,M}^{y,2}(\cdot)$
by the restriction $\tau \in [\lfloor M^2y^2\rfloor +1, m-1]$.

Subdividing the interval $(\lfloor M^2y^2\rfloor, m-1]$ into 
$[\lfloor M^2y^2\rfloor +1, \lfloor m/2\rfloor$], 
$[\lfloor m/2\rfloor +1, \lfloor m(1-\varepsilon_{M,y})\rfloor]$, 
and $[\lfloor m(1-\varepsilon_{M,y})\rfloor +1, m-1]$, 
with $\varepsilon_{M,y}= 1/M^2y^2$, one obtains
\begin{equation*}
\begin{split}
&\P(\Psi_{m,M}^{y,2}([x,x+q))) \lesssim 
\sum_{k=\lfloor M^2y^2\rfloor +1}^{\lfloor m/2\rfloor}
\frac{qy^2}{k^{3/2}m}(\frac{x}{m^{1/2}}\mathrm{e}^{-x^2/2m} + o_m(1)) \\
& +\sum_{k=\lfloor m/2\rfloor +1}^{\lfloor m(1 - \varepsilon_{M,y})\rfloor} 
\frac{qy^2}{m^{3/2}(m-k)}(\frac{x}{(m-k)^{1/2}}\mathrm{e}^{-x^2/2(m-k)} +o_{m-k}(1)) 
+ \sum_{k=\lfloor m(1 - \varepsilon_{M,y})\rfloor +1}^{m-1}
\frac{y^2}{m^{3/2}(m-k)^{1/2}}\\
&\lesssim
\frac{qy}{Mm}
[(\frac{x}{m^{1/2}}\mathrm{e}^{-x^2/2m})(1 + 1/\varepsilon_{M,y}m^{1/2})
+1/q +o_m(1) + \frac{M^3y^3}{m^{1/2}}o_{\varepsilon_{M,y}m}(1)] \\
& \lesssim \frac{qy}{Mm}
[(\frac{x}{m^{1/2}}\mathrm{e}^{-x^2/2m})(1 + o_m(1))
+2/q ] \,, 
\end{split}
\end{equation*}
with $o_m(1)$ being uniform in $x\ge -y$.  
The first inequality uses (\ref{equationABR1}) of 
Lemma \ref{lemmaABRold} together with the
reversed random walk $\tilde{S}_k$ having increments $-X_k$, and 
(\ref{eqA0.1}) and (\ref{eqA0.2}).

So, to demonstrate (\ref{eqA0.3}), it suffices to show
\begin{equation}
\label{eqA3.1}
\lim_{M\rightarrow\infty}
\limsup _{m\rightarrow\infty} \sup_x
|m\, \P(\Psi_{m,M}^{y,1}([x,x+q))) -
\frac{\bar{\beta}_yqx}{m^{1/2}} \mathrm{e}^{-x^2/2m}| = 0.
\end{equation}
Denoting by $\mu_k(\cdot)$ the subprobability measure
$$\mu_k ([w_1, w_2)) = 
\P(S_k< \min_{j\leq k-1}S_j; S_k\in [w_1, w_2))\,,$$
 one has, again using (\ref{eqA0.1}), (\ref{eqA0.2}), and the
 reversed random walk $\tilde{S}_k$,
\begin{equation}
\label{eqA3.2}
\begin{split}
m& \P(\Psi_{m,M}^{y,1}([x,x+q)))=
m\sum_{k=0}^{\lfloor M^2y^2\rfloor} \int_{-y}^0
\Psi_{m-k}^0([x-w,x-w+q))\mu_k(dw) \\
& \quad = Cq(\frac{x}{m^{1/2}}\mathrm{e}^{-x^2/2m} + o_m(1))
\sum_{k=0}^{\lfloor M^2y^2\rfloor}\mu_k([-y,0]) \,,
\end{split}
\end{equation}
for $C>0$ as in (\ref{eqA0.2}), where $o_m(1)$ is uniform in 
$x\in \mathbb{R}_+$.  Hence,  (\ref{eqA3.1}) holds, with 
$\bar{\beta}_y = C\sum_{k=0}^{\infty}\mu_k([-y,0])$ 
(which is $\lesssim 
\sum_{k=1}^{\infty}y^2/k^{3/2} < \infty$), and
$\bar{\beta}_y$ is obviously monotone and right continuous
in $y$.

In order to demonstrate (\ref{eqA0.5}), we further partition 
$\Psi_{m,M}^{y,1}(\cdot)$ into $\Psi_{m,M}^{y,1a}(\cdot)$ and
$\Psi_{m,M}^{y,1b}(\cdot)$, with $\Psi_{m,M}^{y,1a}(\cdot)$
denoting the restriction of $\Psi_{m,M}^{y,1}(\cdot)$ to
$\tau < \lfloor y^2/M^2 \rfloor$ and $\Psi_{m,M}^{y,1b}(\cdot)$
the restriction to $\tau \in  
[\lfloor y^2/M^2 \rfloor, [\lfloor M^2y^2\rfloor ]$.  We note that,
using (\ref{eqA0.1}) and (\ref{eqA0.2}),
\begin{equation}
\label{eqA3.3}
y^{-1}m\,\P(\Psi_{m,M}^{y,1a}([x,x+q))) \lesssim 
\frac{q}{M}(\frac{x}{m^{1/2}}\mathrm{e}^{-x^2/2m} + o_m(1)) \,,
\end{equation}
where $o_m(1)$ is uniform in $x\ge -y$ and
$y\in \mathbb{R}_+$, and hence
\begin{equation*}
\lim_{M\rightarrow\infty} \limsup_{m\rightarrow\infty}
\sup_{x,y\in \mathbb{R}+}
y^{-1}m\,\P(\Psi_{m,M}^{y,1a}([x,x+q))) = 0 \,.
\end{equation*}
So, for (\ref{eqA0.5}), it suffices to demonstrate
\begin{equation}
\label{eqA4.1}
\lim_{M\rightarrow\infty}
\limsup_{y\rightarrow\infty} \limsup_{m\rightarrow\infty}
|y^{-1}m\P(\Psi_{m,M}^{y,1b}([x,x+q))) - 
\bar{\beta}^*q \frac{x}{m^{1/2}}\mathrm{e}^{-x^2/2m}| = 0 
\end{equation}
for some $\bar{\beta}^*>0$ and each $x$.

One proceeds as in (\ref{eqA3.2}), but also applying 
(\ref{eqA0.1}) to $\mu_k(\cdot)$, to obtain
\begin{equation}
\label{eqA4.2}
\begin{split}
y^{-1}m& \, \P(\Psi_{m,M}^{y,1b}([x,x+q)))=
y^{-1}m\sum_{k=\lfloor y^2/M^2\rfloor}^{\lfloor M^2y^2\rfloor} \int_{-y}^0
\Psi_{m-k}^0([x-w,x-w+q))\mu_k(dw) \\
& \quad = C^2q(\frac{x}{m^{1/2}}\mathrm{e}^{-x^2/2m} +o_{m,1}(1))[y^{-1}
\sum_{k=\lfloor y^2/M^2\rfloor}^{\lfloor M^2y^2\rfloor} 
k^{-1/2}(1-\mathrm{e}^{-y^2/2k})+ o_{k,2}(1)]\,,
\end{split}
\end{equation}
where $o_{m,1}(1)$ and $o_{k,2}(1)$ are uniform in $x\ge -y$ and $y\in \mathbb{R}_+$.
(The bound $o_{k,2}(1)$ requires some estimation:  for $k << y^2$ one
employs (\ref{eqA0.2}) as an upper bound for $\mu_k([-y,0])$; for
$k >> y^2$, one employs the upper bound $C'y^2/k^{3/2}$, for
some $C' > 0$.)
Setting $\ell = k/y^2$, one has  
$$y^{-1}\sum_{k=\lfloor y^2/M^2\rfloor}^{\lfloor M^2y^2\rfloor}
k^{-1/2}(1-\mathrm{e}^{-y^2/2k})
\rightarrow \int_0^{\infty} (1-\mathrm{e}^{-1/2\ell})/\sqrt{\ell}\, d\ell$$ 
as first $y\rightarrow\infty$ and then $M\rightarrow\infty$.
Together with (\ref{eqA4.2}), this implies (\ref{eqA4.1}) with 
$\bar{\beta^*} = 
C^2 \int_0^{\infty} (1-\mathrm{e}^{-1/2\ell})/\sqrt{\ell}\, d\ell$.

We now show that (\ref{eq-S-y}) and  
(\ref{eq-S-z}) follow from 
(\ref{eqA0.3}) and (\ref{eqA0.5}).
Denote by $S_k^x$ and $\tilde{S}_k^x$ the random walks with
increments $X_k$ and $-X_k$ started at $x$, and 
abbreviate by setting 
$A_{n,y}^x = \{S_k^x \ge -y \text{ for all } 0 < k < n\}$, with
$A_{n,y} = A_{n,y}^0$. 
We assume $n$ is even; the argument for odd $n$ is the same.  

Setting $m=n/2$, one has, for given $[a_i, b_i)$, with $a_i < b_i$, 
and $N\in \mathbb{Z}_+$, the decomposition
\begin{equation}
\label{eqA5.1}
\begin{split}
&n^{3/2}\, \P(S_n \in [a_i, b_i); A_{n,y}) \\
& \qquad = n^{3/2} (\sum_{x\in -y+\mathbb{Z}_+/N} 
\P(S_m \in [x, x+1/N); A_{m,y}) 
\int_{[x,x+1/N)}\P(S_m^z \in [a_i, b_i); A_{m,y}^z)\nu_m^x(dz))\,,
\end{split}
\end{equation}
where $\nu_m^x(\cdot)$ is some probability measure over $[x,x+1/N)$.
(We will shortly choose 
$[a_i,b_i)$ to be a small subinterval of $[a,b)$.)
As in (\ref{eqjg0.1}),
\begin{equation}
\label{eqA5.2}
\P(S_m^z \in [a_i, b_i); A_{m,y}^z) \in
[\P(\tilde{S}_m^{a_i} \in (z+ a_i -b_i, z]; A_{m,y}^{a_i}), \,
\P(\tilde{S}_m^{b_i} \in (z, z+b_i -a_i]; A_{m,y}^{b_i})] \,.
\end{equation}
Plugging (\ref{eqA5.2}) into (\ref{eqA5.1}), employing 
(\ref{eqA0.2}) and (\ref{eqA0.3}),
and letting $N\rightarrow\infty$, one obtains that the right hand side
of (\ref{eqA5.1}) equals
\begin{equation}
\label{eqA5.3}
2^{3/2}(b_i - a_i)\bar{\beta}_y \hat{\beta}_{i,y}
(\int_0^{\infty} u^2 \mathrm{e}^{-u^2} du + o_m(1)) = 
(1 + o_m(1))2^{1/2}(b_i - a_i)\bar{\beta}_y \hat{\beta}_{i,y},
\end{equation}
where $\hat{\beta}_{i,y} \in [\bar{\beta}_{y+a_i},\bar{\beta}_{y+b_i}]$
and $o_m(1)$ is uniform in $i$.

Now, partition $[a,b)$ by $\{[a_i,b_i)\}_{i=1,\ldots,I}$ so that 
$b_i - a_i \le \varepsilon$, for given $\varepsilon > 0$. By
(\ref{eqA5.1})--(\ref{eqA5.3}),
\begin{equation*}
n^{3/2} \P(S_m \in [a, b); A_{m,y}) =
(1 + o_m(1))2^{1/2}\bar{\beta}_y \sum_i (b_i - a_i) \hat{\beta}_{i,y} \,,
\end{equation*}
which converges to 
$2^{1/2}\bar{\beta}_y \int_a^b \bar{\beta}_{y+v}dv =: \beta_{y,a,b} $  
as $m\rightarrow\infty$ and then 
$\varepsilon \rightarrow 0$.  This implies (\ref{eq-S-y}).  
Clearly, $\beta_{y,a,b}$ is continuous in $a$ and $b$, and, since $\bar{\beta}_y$
is right continuous in $y$, so is $\beta_{y,a,b}$.
The limit
(\ref{eq-S-z}) follows immediately from this and (\ref{eqA0.5}),
with $\beta^* = 2^{1/2} (\bar{\beta}^*)^2$.

The argument for the analog of (\ref{eq-S-y}), but with the restriction $S_k > -y$ 
instead of $S_k \ge -y$, is the same as that for (\ref{eq-S-y}), but 
with $\bar{\beta}_y$ replaced by  
$\bar{\beta}_y^o = C\sum_{k=0}^{\infty}\mu_k((-y,0])$ and
$\beta_{y,a,b}^o:= 2^{1/2}\bar{\beta}_y^o \int_a^b \bar{\beta}_{y+v}^o dv$.
\end{proof}

We now demonstrate (\ref{eq-S-y-gamma}) and (\ref{eq-S-y-new}) of
Lemma \ref{prop-1DRW} in the case where $S_k^{(n)} = S_k$.  Nearly all of the work is devoted to
(\ref{eq-S-y-gamma}); the proof of
(\ref{eq-S-y-new}) will follow quickly by re-applying some of the steps 
for (\ref{eq-S-y-gamma}).  To show (\ref{eq-S-y-gamma}), we partition
time into three intervals, corresponding to the value of 
$\tau$ (as defined above,
the first time at which $S_k$ takes its minimum).  Define the sets
$\Omega_{n,y}$ and $\Omega_{n,y}'$ as in (\ref{eqA6.1}) below.  We will
show that, over each of the three intervals, either 
$\P(\Omega_{n,y} \cap \{\tau =k\})$ and 
$\P(\Omega_{n,y}' \cap \{\tau =k\})$ are close, or 
$\P(\Omega_{n,y}' \cap \{\tau =k\})$ is insignificant.
Hence, $\P(\Omega_{n,y})$ and 
$\P(\Omega_{n,y}')$ will be close.

\begin{proof}[Proof of (\ref{eq-S-y-gamma}) and (\ref{eq-S-y-new}) of
 Lemma \ref{prop-1DRW} for $S_k$]
For both (\ref{eq-S-y-gamma}) and (\ref{eq-S-y-new}), we first restrict ourselves
to the case $a\ge -1/2$, and afterwards show $a < -1/2$ by reversing
the random walk.
We first demonstrate (\ref{eq-S-y-gamma}) and note that
the reasoning for the analog of (\ref{eq-S-y-gamma}), but with "$S_k \ge$" replaced by the strict inequality 
"$S_k >$", will be identical.

Let $\tau$ be defined as above.
Set $I_y = \mathbb{Z}_+ \cap (y^7, n-y^7)$,
$h_n (x) = y^{1/10} + 
h(x \wedge (n-x))$, and
\begin{equation}
\label{eqA6.1}
\begin{split}
\Omega_{n,y} &= \{S_n\in (a, b), S_k\geq -y \mbox { for all } 1\leq k\leq n\}\,,\\
 \Omega'_{n,y}& =\{S_n\in (a, b), S_k\geq -y - h_n(k)\mbox { for all } 1\leq k\leq n\}\,.
\end{split}
\end{equation}
Applying the first part of Lemma \ref{lemmaABRold}, with $y=0$, we obtain 
\begin{align}\label{eq-minimizer-location}
\P(\tau \in I_y; \Omega'_{n,y}) &\le C^2\sum_{k\in I_y} \sum_{j=0}^{y + h_n(k)} 
\frac{j(j+b)}{k^{3/2}(n-k)^{3/2}} \nonumber\\
&\le C^2b_1 \sum_{k\in I_y} \frac{(y +  h_n(k))^3}{k^{3/2} (n-k)^{3/2}} 
\le 2C_1C^2b_1y^{-1/2} n^{-3/2}\,,
\end{align}
where $b_1:=b\vee 1$, $C$ is as in Lemma \ref{lemmaABRold}, and $C_1>0$.  
To prove \eqref{eq-S-y-gamma}, 
we will employ (\ref{eq-minimizer-location}), together with suitably small upper bounds
on $\P(\tau = k; \, \Omega'_{n,y}\setminus \Omega_{n,y})$ relative to $\P(\tau = k; \, \Omega_{n,y})$, 
for every $k\notin I_y$.
Using the symmetry of $h_n(\cdot)$, we only consider $k< y^7$, which we break into the subcases $k\leq y^{19/10}$
and $k\in (y^{19/10}, y^7)$. 

For $k\leq y^{19/10}$, 
\begin{align}\label{eq-small-k-prime}
&\P(\tau = k; \Omega'_{n,y} \setminus \Omega_{n,y})\nonumber\\
 \leq& \P(-y - h_n(k)\leq S_k \leq -y) \max_{-y^*-h_n(k)\leq x\leq -y}\P(S_j\geq 0 \mbox{ for all } 1\leq j\leq n-k; S_{n-k} \in (a-x, b-x))\nonumber\\
 \leq & C_2C y^{-2} \cdot (b+y) n^{-3/2}\le 2C_2Cb_1 y^{-1}n^{-3/2} \,,
\end{align}
where $C_2>0$.
The second term of the first inequality, 
on the third line, follows from (\ref{equationABR1}) of Lemma
\ref{lemmaABRold} with $y=0$, 
and the first term follows 
from the assumption of exponential moments on the random walk (40 moments suffices) together 
with a moment estimate.  
On the other hand, by (\ref{equationABR2}) of Lemma \ref{lemmaABRold} (for large enough $n$),
$$\P(\tau = k; \,\Omega_{n,y}) \gtrsim n^{-3/2}\sum_{i=1}^{\sqrt{k}} \frac{i (i+a)}{k^{3/2}} 
\asymp (1+b_1/\sqrt{k}) n^{-3/2} \ge (1+b_1/y)n^{-3/2}\,.$$
Combined with \eqref{eq-small-k-prime}, it follows that, for $k\leq y^{19/10}$,
\begin{equation}\label{eq-ratio-small-k}
\P(\tau = k; \,\Omega'_{n,y} \setminus \Omega_{n,y}) \lesssim y^{-1}\P(\tau = k; \,\Omega_{n,y})\,.
\end{equation}

We next consider $k\in (y^{19/10}, y^7)$.  Since, for appropriate $C''$,  
the boundary for $\Omega'_{n,y}$ lies at most
$C''y^{1/10}$ below that for $\Omega_{n,y}$ at any time $k$, another application of (\ref{equationABR1})
implies that
\begin{equation}
\label{eq-postratio}
\begin{split}
\P(\tau = k; \Omega'_{n,y} \setminus \Omega_{n,y}) 
&\le C'''y^{1/10}\cdot k^{-3/2} (y+ \log k)(y+ \log k+b)n^{-3/2} \\
&\le 4C'''k^{-3/2}y^{11/10}(y+b) n^{-3/2}\,,
\end{split}
\end{equation}
for appropriate $C'''$.
Another application of (\ref{equationABR2}), after summation, implies that
\begin{equation}
\label{eqpostnew1}
\P(\tau = k; \,\Omega_{n,y}) \gtrsim (y\wedge \sqrt{k})^2 (y\wedge \sqrt{k}+b)k^{-3/2}n^{-3/2}\,.
\end{equation}
Combining the last two displays, it follows that, for $k\in (y^{19/10}, y^7)$,
\begin{equation}
\label{equationBTA}
\P(\tau = k; \,\Omega'_{n,y} \setminus \Omega_{n,y}) \lesssim y^{-7/10}\P(\tau = k; \,\Omega_{n,y})\,.
\end{equation}
Also note that summing (\ref{eqpostnew1}) over $k\in [y^2,2y^2]$
implies 
\begin{equation}
\label{eqpostnew2}
\P(\Omega_{n,y}) \gtrsim y(y+b)n^{-3/2},
\end{equation}
which dominates the bound in \eqref{eq-minimizer-location} for large $y$.

Together, \eqref{eq-minimizer-location}, \eqref{eq-ratio-small-k}, and (\ref{equationBTA}) 
imply that
\begin{equation}
\label{eqpostnew3}
\P(\Omega'_{n,y}) \lesssim (1+y^{-7/10}) \P(\Omega_{n,y}) + b_1 y^{-1/2} n^{-3/2}\,.
\end{equation}
Along with (\ref{eqpostnew2}) and \eqref{eq-S-y}, 
this completes the proof of \eqref{eq-S-y-gamma}.

The inequality (\ref{eq-S-y-new}), for $j=n$ and $a\ge -1/2$,
follows by applying the first part of Lemma 
\ref{lemmaABRold} (with general $y$) to bound $\P(\Omega_{n,y})$ from above, 
and then combining 
this with the upper bounds \eqref{eq-minimizer-location}, \eqref{eq-small-k-prime}, and (\ref{eq-postratio}), for 
$\P(\tau = k; \Omega'_{n,y} \setminus \Omega_{n,y})$ over the three 
ranges of $k$.  Specifically, the bound in (\ref{eq-S-y-new})
is of the same order as that in (\ref{equationABR1}) and, in
place of the $y$ coefficients in (\ref{equationABR1}),  the bound in \eqref{eq-minimizer-location}
contributes the coefficient $b_1 y^{-1/2}$, the bound in 
\eqref{eq-small-k-prime} contributes the coefficient $b_1 y^{9/10}$
(after summing over the region $k\le y^{19/10}$), and the bound
in  (\ref{eq-postratio}) contributes the coefficient $y^{3/20}(y+b)$
(after summing over the region $k\in (y^{19/10},y^7)$).

In order to show (\ref{eq-S-y-gamma}) and (\ref{eq-S-y-new}) for $a < -1/2$, it 
suffices to also assume that $b\le 0$ and $b-a \le 1/2$.   Denoting by 
$\tilde{\Omega}_{n,y}$ and $\tilde{\Omega}'_{n,y}$ the analogs of the sets
$\Omega_{n,y}$ and $\Omega'_{n,y}$, but for the reversed random walk
$\tilde{S}_k$ rather than $S_k$, and using (\ref{eqjg0.1}), 
for (\ref{eq-S-y-gamma}) it is enough to show that
\begin{equation}
\label{eq1112.1}
\P(\tilde{\Omega}'_{n,y+1/2}\setminus\tilde{\Omega}_{n,y-1/2}) \le \delta_y \P(\tilde{\Omega}_{n,y-1/2}) 
\end{equation}
for the interval $(\tilde{a},\tilde{b})$, with $\tilde{a} = -b \ge 0$ and $\tilde{b} = -a$, where $\delta_y \rightarrow 0$ as $y\rightarrow\infty$.
One can employ the same reasoning as for the case $a\ge -1/2$, with the bounds 
in \eqref{eq-minimizer-location} --(\ref{eq-postratio}) holding for
different constants in front.  (The term $h(k)$ there needs to be increased
by $1$.)  The bound (\ref{eq1112.1}) follows after combining these inequalities as in (\ref{eqpostnew3}).  For (\ref{eq-S-y-new}) 
with $j=n$ and $a< -1/2$, one employs
the same reasoning as in the previous paragraph, but for $\tilde{S}_k$
instead of $S_k$.
\end{proof}

\subsection{Demonstration of Lemmas \ref{lemmaABR} and
\ref{prop-1DRW} for general $d^{(n)}$}

Let the random walks $\{S_k\}_{k=0,\ldots,n}$ and 
$\{S_k^{(n)}\}_{k=0,\ldots,n}$
be as in Lemmas \ref{lemmaABR} and \ref{prop-1DRW}.
In this subsection, we show that the non-crossing probabilities 
of the curves there are asymptotically the same 
as $n\rightarrow\infty$, which enables us to show 
Lemmas \ref{lemmaABR} and
\ref{prop-1DRW} for general $d^{(n)}$.  
We will find it convenient to consider two choices of 
translation terms $d^{(n,i)}$ satisfying $d^{(n,1)} < d^{(n,2)}$
and $|d^{(n,i)}| \le c(\log n)/n$.  Also, set 
$h^{(n,i)}(k) = h(k \wedge (n-k)) - d^{(n,i)} k $, 
$a^{(n,i)} = a + d^{(n,i)}n$ and $b^{(n,i)} = b + d^{(n,i)}n$.
\begin{lemma}
\label{lemjg}
Let $S_k$, $c$, and $d^{(n,i)}$, $i=1,2$, 
be as above.  
Then, for fixed $\varepsilon > 0$, $\delta > 0$, and appropriate $C > 0$,
\begin{equation}
\label{eqjg1.1}
\begin{split}
& \P(S_n \in (a^{(n,2)},b^{(n,2)}), S_k\ge -y-h^{(n,2)}(k)
\text{ for all } 0<k<n) \\
&\qquad \le (1+\delta) \P(S_n \in 
(a^{(n,1)}-\delta,b^{(n,1)}+\delta), S_k > -y-h^{(n,1)}(k) \text{ for all } 0<k<n)  \\
&\qquad \qquad  + C(y \vee 1)((y+a)\vee 1)/n^{25/16}
\end{split}
\end{equation}
for all $n$, $y\ge 0$, and $-y\le a < b < \infty$ with $b-a = \varepsilon$.
\end{lemma}

We will demonstrate Lemma \ref{lemjg} at the end of this subsection.
In the following corollary, $h^{(n)}(k)$, $a^{(n)}$, and  $b^{(n)}$ are the
analogs of $h^{(n,i)}(k)$, $a^{(n)}_i$, and  $b^{(n)}_i$ 
with a given $d^{(n)}$.

\begin{cor}
\label{corjg}
Suppose that $d^{(n)}$ is as above and $d^{(n)} >0$ for all $n$.  
For any $y\geq 1$ and $1-y\leq a <b<\infty$, 
\begin{equation}
\label{eqjg1.3}
\lim_{n\rightarrow\infty} \frac{\P(S_n\in (a^{(n)}, b^{(n)}), 
S_k\geq -y- h^{(n)}(k) \mbox { for all } 0 < k < n)} 
{\P(S_n\in (a, b), 
S_k > -y - h(k \wedge (n-k))\mbox { for all } 0 < k < n)} = 1\,,
\end{equation}
with the rate of convergence being uniform over all sequences
$d^{(n)}$ satisfying $|d^{(n)}| \le c(\log n)/n$ for given $c>0$.
If, instead, $d^{(n)} <0$ is assumed for all $n$,
then the analog of (\ref{eqjg1.3}) holds, but with 
``$\,S_k >\,$" replaced by ``$\,S_k \ge \,$" in the denominator.
The same limits hold, in each case,  
if ``$\,S_k \ge $" is replaced by ``$\,S_k >\,$"
in the numerator.  
\end{cor}
\begin{proof}
Since the proofs of all four statements are similar, 
we prove just (\ref{eqjg1.3}). 
The upper bound $1$ for the limit on the left hand side of 
(\ref{eqjg1.3}) is obtained
by setting $d^{(n,1)} = 0$ and $d^{(n,2)} = d^{(n)}$ in 
(\ref{eqjg1.1}), and then employing the
continuity of $\beta_{y,a,b}$ in $a$ and $b$, together with the
lower bound given by (\ref{eq-S-y}) for $d^{(n,1)} = 0$, which
decays more slowly than $n^{-25/16}$.  

The lower
bound $1$ is obtained by applying (\ref{eqjg1.1}) to 
$\tilde{S}_k$ and reversing the roles of  $d^{(n,1)}$ and
$d^{(n,2)}$.   
After partitioning $(a,b)$ into subintervals with endpoints 
$a_i$ and $b_i$ satisfying $0\le b_i - a_i \le 1/N$, for given
$N\in \mathbb{Z}_+$, the analog of the
lower bound in (\ref{eqjg0.1}) implies that
\begin{equation}
\label{eqjg1.4}
\begin{split}
& \P(\tilde{S}_n\in [-b^{(n)}_i, -a^{(n)}_i), 
\tilde{S}_k > -y- \hat{h}^{(n)}(k) - a_i \mbox { for all } 0 < k < n) \\
& \quad \le \P(S_n\in (a^{(n)}_i, b^{(n)}_i], 
S_k \ge -y - h^{(n)}(k)  \mbox { for all } 0 < k < n)\,,
\end{split}
\end{equation}
where $\hat{h}^{(n)}(k)$ is the
analog of $h^{(n)}(k)$, but for the translation
$-d^{(n)}$ instead of $d^{(n)}$.
(For the terminal subinterval with $b_i = b$, we instead use $(a^{(n)}_i,b^{(n)}_i)$.)
On the other hand, by (\ref{eqjg1.1}), for any $\delta > 0$,
\begin{equation}
\label{eqjg1.5}
\liminf_{n\rightarrow\infty}\frac{\P(\tilde{S}_n\in [-b^{(n)}_i, -a^{(n)}_i), 
\tilde{S}_k > -y- \hat{h}^{(n)}(k) - a_i \mbox { for all } 0 < k < n)} 
{\P(\tilde{S}_n\in [-b_i +\delta, -a_i - \delta), 
\tilde{S}_k \ge -y - h(k \wedge (n-k)) - a_i\mbox { for all } 0 < k < n)}
\ge 1\,,
\end{equation}
with convergence being uniform over all 
sequences satisfying $|d^{(n)}| \le c(\log n)/n$.  
Application of the analog of the upper bound in 
(\ref{eqjg0.1}) implies that the denominator in (\ref{eqjg1.5}) is
at least the denominator of (\ref{eqjg1.3}), but with the interval
$(a, b)$ there 
replaced by $(a_i +\delta, b_i - \delta]$ and the term
$h(k \wedge (n-k))$ replaced by $h(k \wedge (n-k)) - 1/N$.
Together with (\ref{eqjg1.4}), (\ref{eqjg1.5}), and the continuity
of $\beta_{y,a,b}^o$ in $a$ and $b$, this implies
\begin{equation}
\label{eqjg1.6}
\liminf_{n\rightarrow\infty} \frac{\P(S_n\in (a^{(n)}_i, b^{(n)}_i], 
S_k\geq -y- h^{(n)}(k) \mbox { for all } 0 < k < n)} 
{\P(S_n\in (a_i, b_i], 
S_k > -y - h(k \wedge (n-k)) +1/N \mbox { for all } 0 < k < n)} \ge 1
\end{equation}
uniformly in $d^{(n)}$.
Summation of the probabilities over all $i$ in the numerator and
in the denominator of (\ref{eqjg1.6}), and then letting 
$N\rightarrow\infty$ gives the lower bound in (\ref{eqjg1.3}), since
$\beta_{y,a,b}^o$ is left continuous in $y$. 
\end{proof}

We next show that both Lemma \ref{lemmaABR}
and Lemma \ref{prop-1DRW} follow from  Corollary \ref{corjg} and the restricted version of Lemma 
\ref{prop-1DRW} with $d^{(n)} \equiv 0$.
Recall that the random walk
$\{S_k^{(n)}\}_{k=0,\ldots,n}$
is defined in Lemma \ref{lemmaABR} and satisfies $\E^{(n)}(S_k^{(n)})=0$.

\begin{proof}[Proof of Lemmas \ref{lemmaABR} and \ref{prop-1DRW}]
Let $\hat{a}^{(n)}$, $\hat{b}^{(n)}$, and  $\hat{h}^{(n)}(k)$ be the
analogs of $a^{(n)}$, $b^{(n)}$, and  $h^{(n)}(k)$, but for the translation
$-d^{(n)}$ instead of $d^{(n)}$.  
In order to show (\ref{eq-S-y}) --(\ref{eq-S-y-gamma})
of Lemma \ref{prop-1DRW}, we note that
\begin{equation}
\label{eqjg2.1}
\begin{split}
& \P^{(n)}(S_n^{(n)}\in (a, b), 
S_k^{(n)}\geq -y- h(k\wedge (n-k)) \mbox { for all } 0 < k < n) \\
& \quad = \P^{(n)}(S_n \in (\hat{a}^{(n)}, \hat{b}^{(n)}), 
S_k\geq -y- \hat{h}^{(n)}(k) \mbox { for all } 0 < k < n) \\
& \quad = \gamma_{a,b,d,y,h}^{(n)} \P(S_n \in (\hat{a}^{(n)}, \hat{b}^{(n)}), 
S_k\geq -y- \hat{h}^{(n)}(k) \mbox { for all } 0 < k < n),
\end{split}
\end{equation}
where $\gamma_{a,b,d,y,h}^{(n)}$ is bounded above by
$\mathrm{e}^{-\theta^{(n)}\hat{a}^{(n)}}$
and bounded below by 
$\mathrm{e}^{-\theta^{(n)}\hat{b}^{(n)}}/\mathbb{E}({\mathrm{e}^{-\theta^{(n)}S_n }})$, 
and $\theta^{(n)}$ is as in (\ref{eq-change-of-measure-0}).
Since $|d^{(n)}|\le c(\log n)/n$, one can check that 
$|\theta^{(n)}|\le c'(\log n)/n$ for some $c'>0$.
Since $a$ and $b$ are fixed, $\E(S_n) = 0$, and $S_n$ has exponential
moments, it follows that both bounds converge to $1$ as
$n\rightarrow\infty$ uniformly in $d^{(n)}$.  Together with (\ref{eqjg1.3}) (with $-d^{(n)}$
in place of $d^{(n)}$) and the restricted versions of 
(\ref{eq-S-y}) --(\ref{eq-S-y-gamma}) for $d^{(n)} = 0$, this implies 
(\ref{eq-S-y}) --(\ref{eq-S-y-gamma}) for general $d^{(n)}$.

We next show (\ref{eq-S-y-new}) of Lemma \ref{prop-1DRW},
first restricting ourselves to the case where $j=n$.  Using
the above bound on $|\theta^{(n)}|$, when $a> n/\log n$, it is not
difficult to show (\ref{eq-S-y-new}) by ignoring the boundary and 
using the same moderate
deviation inequality as in (\ref{eq23here}).  When $a\le n/\log n$,
we first
consider the case where $d^{(n)} <0$.  Then (\ref{eqjg2.1}) again holds and, 
using the above bound on $|\theta^{(n)}|$,
\begin{equation}
\label{eqjg2.2}
\begin{split}
& \P^{(n)}(S_n^{(n)}\in (a, b), 
S_k^{(n)}\geq -y- h(k\wedge (n-k)) \mbox { for all } 0 < k < n) \\
& \quad \lesssim \P(S_n \in (\hat{a}^{(n)}, \hat{b}^{(n)}), 
S_k\geq -y- \hat{h}^{(n)}(k) \mbox { for all } 0 < k < n).
\end{split}
\end{equation}
Together with (\ref{eqjg1.1}) (with $-d^{(n)}$ in place of $d^{(n)}$) 
and the restricted version of (\ref{eq-S-y-new}), this implies the general
version of (\ref{eq-S-y-new}) for $d^{(n)} < 0$.  When
$d^{(n)} >0$, we employ the reversed random walk $\tilde{S}_k$ and 
note that, as in (\ref{eqjg0.1}),
\begin{equation}
\label{eqjg2.3}
\begin{split}
& \P^{(n)}(S_n^{(n)}\in (a, b), 
S_k^{(n)}\geq -y- h(k\wedge (n-k)) \mbox { for all } 0 < k < n) \\
& \quad \le \P^{(n)}(\tilde{S}_n^{(n)}\in (-b, -a), 
\tilde{S}_k^{(n)}\geq -y- h(k\wedge (n-k)) -b \mbox { for all } 0 < k < n) 
\end{split}
\end{equation}
for any $a$, $b$, and $y$.
Since $-{d}^{(n)} < 0$, one can apply (\ref{eqjg2.2}), with $\tilde{S}_k^{(n)}$ in place of $S_k^{(n)}$, and then reason as above.

The demonstration of (\ref{eq-S-y-new}) for general $j \ge n/2$ requires
just a slight modification of the argument in the previous paragraph.  
Since $h(\cdot)\geq 0$ is increasing and
concave, 
we have 
\begin{equation}
        \label{eq-march16a}
h(k\wedge (n-k)) \le 2h(k\wedge (j-k)) + (k/j)h(n-j) 
\quad \text{for all } 0\le k\le j\,.
\end{equation}
(This is the only point in the proof of (\ref{eq-S-y-new}) where the
monotonicity and concavity of $h(\cdot)$ are used.)
Indeed, the inequality in \eqref{eq-march16a} holds trivially when
$k\leq j/2$ and  it holds with equality when $k=j$; for $j/2<k<j$,
it follows from the inequalities
$$ h(k\wedge (n-k))-h(n-j)\leq h(n-k)-h(n-j)\leq h(j-k)$$
and
$$\frac{h(j-k)}{j-k}\geq \frac{h(j)}{j}\geq \frac{h(n-j)}{j}\,,$$
where $j\geq n/2$
 was used in the last inequality
(multiply the latter display by $j-k$ and add $h(j-k)$ to both sides, 
before applying it to the preceding display).
Setting $d_2^{(n)} = 2C' (\log \lfloor n/2 \rfloor)/n$
and $h_2^{(n,j)}(k) = 2h(k\wedge (j-k)) + d_2^{(n)}k$, it follows that
the probability on the left hand side of (\ref{eq-S-y-new}) is, for any
$j\in [n/2,n]$, at most 
\begin{equation}
\label{eqjg2.4}
\P^{(n)}(S_j^{(n)}\in (a, b), 
S_k^{(n)}\geq -y- h_2^{(n,j)}(k) \mbox { for all } 0 < k < j)\,. 
\end{equation}
One can now apply the same reasoning as for the left hand side
of (\ref{eqjg2.3}), but stopping the process $S^{(n)}_k$ at
time $j$ instead of $n$ and applying the additional tilting 
induced by $d_2^{(n)}$.  The bound here, as above, is uniform up to 
the choice of $c$.
This concludes the proof of Lemma \ref{prop-1DRW}.

The equation (\ref{equationABR1}) of Lemma \ref{lemmaABR} is a special case of
(\ref{eq-S-y-new}), with $h(k) \equiv 0$.   So, to prove Lemma \ref{lemmaABR},
we need only still to show (\ref{equationABR2}).  
Supposing that $d^{(n)}>0$,  
we again apply (\ref{eqjg2.1}).  Since $a\le \sqrt{n}$,
$\gamma_{a,b,d,0,0} \ge C'$ for some constant $C'>0$.  
Using (\ref{eqjg1.1}), with $-d^{(n)}$ in place of $d^{(n,1)}$
and $d^{(n,2)} \equiv 0$,  (\ref{equationABR2}) follows from this 
and the special case
of (\ref{equationABR2}) with $d^{(n)} = 0$.  For $d^{(n)}<0$, similar
reasoning holds after employing the reversed random walk $\tilde{S}^n_k$
and (\ref{eqjg0.1}).
\end{proof}

We now demonstrate Lemma \ref{lemjg}.  

\begin{proof}[Proof of Lemma \ref{lemjg}]
In order to show (\ref{eqjg1.1}), we decompose the interval $[0,n]$ into three parts,
$[0,e_1]$, $[e_1,e_2]$, and $[e_2,n]$, where $e_1 = n- 
\lfloor n^{1/4} \rfloor- \lfloor n^{1/12}\rfloor$ and
$e_2 = n-\lfloor n^{1/4}\rfloor$.  
(We could choose larger powers of $n$, e.g., $n^{3/4}$ instead
of $n^{1/4}$ and $n^{1/4}$ instead of $n^{1/12}$.  However, these smaller
powers of $n$ are required in the lattice setting
of Section \ref{sec-lattice}, and so we also 
employ them here for conformity.)
To compare the probabilities 
on the left and right
hand sides of (\ref{eqjg1.1}), we will proceed in essence as follows:\\
$\bullet$
(A) Over the first interval
$[0,e_1]$, compare the probabilities that the same path on each of the two sides always lies
above the corresponding boundary.   Since $d^{(n,2)} \ge d^{(n,1)}$, the boundary on the left hand
side is higher than that on the right hand side, and so the inequality is automatic on this interval.  \\
$\bullet$
(B) Over the middle interval $[e_1,e_2]$, compare a path on the right hand side with the path on the
left hand side that at time $e_1$ takes the same value, but at time $e_2$ is larger by 
$\ell_{1,2}^{(n)}:=(d^{(n,2)}-d^{(n,1)})e_2$.  
We will then employ a version of the local central limit theorem at
time $e_2$ to 
compare probabilities for corresponding paths; 
since $d_{1,2}^{(n)}:=d^{(n,2)}-d^{(n,1)}$ is of order $(\log n)/n$, the probabilities will be approximately the same.  \\
$\bullet$
(C)  Over the last
interval $[e_2,n]$, compare paths that are identical over the interval, except 
for the translation $\ell_{1,2}^{(n)}$ inherited from the middle interval.  After also translating the boundaries 
by $\ell_{1,2}^{(n)}$, the boundary on the left hand
side lies above the boundary on the right hand side.   If the value taken at time $n$ by the path
on the left hand side lies in $(a^{(n,2)}, b^{(n,2)})$, the value on the right hand side lies in $(a^{(n,1)}, b^{(n,1)}+ \delta)$ if 
$\delta \ge d_{1,2}^{(n)}(n - e_2) = d_{1,2}^{(n)}\lfloor n^{1/4}\rfloor = o(1)$. 

\noindent
$\bullet$
The inequality (\ref{eqjg1.1}) will then follow by combining (A)--(C) 
using the Markov property of random walk.

As indicated above, the inequality that is employed for (A) is immediate:
for $a_1 \le b_1$,
\begin{equation}
\label{eqjg3.1}
\begin{split}
& \P(S_{e_1} \in (a_1,b_1), S_k\ge -y-h^{(n,2)}(k)
\text{ for all } 0<k\le e_1) \\
&\qquad \le \P(S_{e_1} \in 
(a_1,b_1), S_k > -y-h^{(n,1)}(k) \text{ for all } 0<k\le e_1)\,.  
\end{split}
\end{equation} 
The inequality that is employed for (C) is also immediate: 
for fixed $\delta \ge d_{1,2}^{(n)}\lfloor n^{1/4}\rfloor$ 
and any $a\le b$ and $x_2$,
\begin{equation}
\label{eqjg3.2}
\begin{split}
& \P(S_{\lfloor n^{1/4}\rfloor} \in (a^{(n,2)} - \ell_{1,2}^{(n)} ,b^{(n,2)} - \ell_{1,2}^{(n)}), 
S_k\ge -y-h^{(n,2)}(k') - x_2 - \ell_{1,2}^{(n)},\,\,
0<k\le \lfloor n^{1/4}\rfloor) \\
&\qquad \le \P(S_{\lfloor n^{1/4}\rfloor} \in 
(a^{(n,1)},b^{(n,1)} + \delta), S_k > -y-h^{(n,1)}(k') - x_2\,,\,\, 
 0<k\le \lfloor n^{1/4}\rfloor),  
\end{split}
\end{equation} 
with $k' := \lfloor n^{1/4} \rfloor - k$, 
since $\ell_{1,2}^{(n)} = d_{1,2}^{(n)}n - d_{1,2}^{(n)}\lfloor n^{1/4}\rfloor$ 
with
$h^{(n,2)}(k') + \ell_{1,2}^{(n)}  \le h^{(n,1)}(k')$ on the interval.

We still need to obtain bounds corresponding to (B), for which we 
first need to restrict
the range of the values at the initial and terminal points of the interval
and to obtain several related bounds.  We
will show that, for given $\varepsilon > 0$, $\delta > 0$, and large enough $n$, 
\begin{equation}
\label{eqjg3.3}
\P(S_{\lfloor n^{1/12}\rfloor} \in (a_2 + \ell_{1,2}^{(n)}, b_2 + \ell_{1,2}^{(n)}))
\le  (1+ \delta)\P(S_{\lfloor n^{1/12}\rfloor} \in 
(a_2,b_2))
\end{equation} 
for all $a_2$ and $b_2$ satisfying $\varepsilon = b_2 - a_2$ and 
$|a_2|\le n^{1/16}$.  On the other hand, 
it follows from a moderate deviation estimate that,
for any $M$ and large enough $C$,
\begin{equation}
\label{eqjg4.1}
\P(|S_{k^{1/12}}| > \tfrac{1}{2}n^{1/16} \text{ for some } 0<k\le n^{1/12}) \le Cn^{-M}
\end{equation}
for all $n$.  
By applying (\ref{eq-S-y-new}) with $j=n$ in the case $d^{(n)} \equiv 0$, 
over both $[0,e_1]$ and $[e_1,n]$, 
it moreover follows that
\begin{equation}
\label{eqjg4.2}
\begin{split}
& \P(S_n \in (a^{(n,2)},b^{(n,2)}), S_{e_1} 
\le n^{1/12}, S_k\ge -y-h^{(n,2)}(k) 
\text{ for all } 0<k < n)  \\
& \qquad \quad 
\le C_1 (y\vee 1)((y+a^{(n,2)})\vee 1) n^{-13/8}
\le C'_1 (y\vee 1)((y+a)\vee 1) n^{-25/16}
\end{split}
\end{equation}
for  large enough $C_1$ and $C'_1$ depending on $c$, and any 
$n$, $y\ge 0$, 
and $a$ and $b$ with
$b-a >0$ fixed.  

The inequality (\ref{eqjg1.1}) follows from the 
inequalities (\ref{eqjg3.1})--(\ref{eqjg4.2}):  Dividing $[0,n]$
into the three subintervals $[0,e_1]$, $[e_1, e_2]$, and $[e_2,n]$ defined
earlier, we will apply the Markov property to $S_k$, 
letting $S_{e_1} \in [x_1,x_1 + dx)$ and 
$S_{e_2} \in [x_2,x_2 + dx)$, for given $x_1$ and $x_2$, and then
integrating over $x_1$ and $x_2$.  On account of
(\ref{eqjg4.2}), we may restrict attention to $x_1 \ge n^{1/12}$. 
Restarting the process $S_k$ at time $e_1$ and applying
 (\ref{eqjg4.1}), we may also disregard paths that cross the boundary
over $[e_1,e_2]$, which enables us to apply (\ref{eqjg3.3}).  
Together with (\ref{eqjg3.1})--(\ref{eqjg3.2}), this
implies (\ref{eqjg1.1}). 

To complete the proof of (\ref{eqjg1.1}), we still need to verify
(\ref{eqjg3.3}).  We introduce
two variables, $S^{(n,1)}$ and $S^{(n,2)}$, with 
$S^{(n,1)} = S_{\lfloor n^{1/12}\rfloor} - a_2$ and 
$S^{(n,2)} = S_{\lfloor n^{1/12}\rfloor} - a_2 -\ell_{1,2}^{(n)}$, and two measures
$\P^{(n,1)}$ and $\P^{(n,2)}$, defined as in 
(\ref{eq-change-of-measure-0}) by tilting the measure $\P$ with
appropriate $\theta^{(n,1)}$ and $\theta^{(n,2)}$, so that
the corresponding expectations satisfy 
$\E^{(n,2)}(S^{(n,2)}) = \E^{(n,1)}(S^{(n,1)}) = 0$.  Reasoning as in 
(\ref{eqjg2.1}),  one can show that
\begin{equation}
\label{eqjg5.1}
\begin{split}
\P(S_{\lfloor n^{1/12}\rfloor} \in (a_2 + \ell_{1,2}^{(n)}, b_2 + \ell_{1,2}^{(n)})
&= \gamma^{(n)}_{a_2,b_2,d} \P^{(n,2)}(S^{(n,2)}\in (0,\varepsilon)), \\
\P(S_{\lfloor n^{1/12}\rfloor} \in (a_2, b_2)
&= \gamma^{(n)}_{a_2,b_2} \P^{(n,1)}(S^{(n,1)}\in (0,\varepsilon)), 
\end{split}
\end{equation}
with $\gamma^{(n)}_{a_2,b_2}, \gamma^{(n)}_{a_2,b_2,d}$ satisfying    
$1 - C n^{-1/50} \le
 \gamma^{(n)}_{a_2,b_2,d}/\gamma^{(n)}_{a_2,b_2}
\le 1 + C n^{-1/50} $, for appropriate $C>0$ depending on $c$ but not
on $a_2$.  
(Note that $\ell_{1,2}^{(n)} \le C' \log n$,
$|\theta^{(n,i)}| \le C'n^{-1/48}$, and 
$|\theta^{(n,2)}-\theta^{(n,1)}| \le C'(\log n)n^{-1/12}$ for
appropriate $C'$, and so, for $|a_2| \le n^{1/16}$,   
$$\exp \{\theta^{(n,2)}|(a_2+ \ell_{1,2}^{(n)})-\theta^{(n,1)}a_2|\}
\le C'' (\log n) n^{-1/48}$$ for appropriate $C''$, with $C'$ and $C''$ depending on $c$.)
Therefore, to demonstrate (\ref{eqjg3.3}), it suffices to show that
\begin{equation}
\label{eqjg5.2}
\P^{(n,2)}(S^{(n,2)}\in (0,\varepsilon))/\P^{(n,1)}(S^{(n,1)}\in (0,\varepsilon))
\rightarrow 1 \qquad \text{as } n\rightarrow\infty,
\end{equation}
uniformly in $|a_2| \le n^{1/16}$.

In order to show (\ref{eqjg5.2}), we will 
apply the central limit asymptotic
expansion given in Feller [\cite{Fell}, Theorem 16.4.1], which states that
\begin{equation}
\label{eqjg5.3}
F_n(x) -\mathfrak{N}(x) - \frac{\mu_3}{6\sigma^3 \sqrt{n}} (1-x^2)\mathfrak{n}(x) = o(\frac{1}{\sqrt{n}}),
\end{equation}
with $F_n(x) := F^{*n}(x\sigma \sqrt{n})$,  where 
$F^{*n}(\cdot)$ is the $n$-fold convolute of a nonlattice, 
mean $0$ random 
variable with variance $\sigma^2$ and third moment $\mu_3$, and 
$\mathfrak{N}(\cdot)$ and $\mathfrak{n}(\cdot)$ 
denote the distribution and
density of a standard normal random variable.  It is not difficult to show that
the variance and third moments of the summands 
$X_k^{(n,i)}$ of $S^{(n,i)}$ (with respect to
$\P^{(n,i)}$), $i=1,2$, converge uniformly over $|a_2| \le n^{1/16}$ to the
variance and third moment of $X_k$.
One can use this to show 
that the error on the right hand 
side of (\ref{eqjg5.3}) is uniform when the formula is applied to 
$S^{(n,i)}$ over this range of $a_2$; 
the limit (\ref{eqjg5.2}) will then follow
from (\ref{eqjg5.3}).   
We summarize the arguments for these steps in the next two paragraphs.

Let $\mu_k^{(i)}$ denote the $k$th moment of $X_k^{(n,i)}$.
Since $|\theta^{(n,i)}| \le C' n^{-1/48}$, it is
not difficult to show that 
\begin{equation}
\label{eqjg5.4}
|\mu_k^{(i)} - \mu_k| \le C_k n^{-1/48}, \qquad i=1,2,
\end{equation}
for any $k$ and appropriate $C_k$, for all $|a_2| \le n^{1/16}$. 
The bound (\ref{eqjg5.4}) will be used to adapt the proof of Theorem 16.4.1
to our setting.

The argument for Theorem 16.4.1 consists of comparing $F_n(\cdot)$
with the remaining terms on the left hand side of (\ref{eqjg5.3})
(denoted by $G(x)$ in Feller \cite{Fell}) by
applying the inversion inequality (16.4.4) of \cite{Fell}, which integrates the
difference of their respective characteristic
functions over an appropriate interval.
In the proof of Theorem 16.4.1, there are three  
contributions to the error term on the right hand side 
of (\ref{eqjg5.3}):  (a) a truncation error $\epsilon$ that arises by
restricting the interval of integration to that in (16.4.4) and that depends 
on $G(\cdot)$; (b) a rapidly decreasing error term in $n$ that depends
on the lower bound of the difference between $1$
and the maximum of the 
absolute value of the characteristic function of $F_n (\cdot)$, 
on an appropriately chosen subinterval of the interval of integration 
in (16.4.4); and (c) a three-term Taylor expansion
for the characteristic function of $F_n(\cdot)$ that 
is compared with the characteristic function of $G(\cdot)$ over
a third subinterval.
Employing (\ref{eqjg5.4}), with $k=2,3$, it is not difficult to
check that one obtains uniform bounds on the 
errors in (a)--(c) for $S^{(n,i)}$ over
$|a_2| \le n^{1/16}$.
Hence, the analog of (\ref{eqjg5.3}) will also hold, with 
the uniform error bound $o(1/\sqrt{n})$ over $|a_2| \le n^{1/16}$.
The limit (\ref{eqjg5.2}) follows from this error bound and another application 
(\ref{eqjg5.4}), which is applied to the left hand side of (\ref{eqjg5.3}) .  
This completes the
demonstration of (\ref{eqjg3.3}) and hence the proof of
Lemma \ref{lemjg}.
\end{proof}  

\smallskip
\noindent
{\bf Acknowledgment} We thank an anonymous referee for a careful reading of the first version
of the paper, and for useful comments.

%\small
\bibliographystyle{abbrv}
%\bibliography{mbrw}

\end{document}